\documentclass[12pt,twoside,a4paper]{amsart}
\usepackage{amsfonts}
\usepackage{txfonts}
\usepackage{bbm}
\usepackage{mathrsfs}
\usepackage{amsfonts}
\usepackage{amsmath}
\usepackage{amssymb}
\usepackage{wasysym}
\usepackage{psfrag}
\usepackage{graphics,epsfig,amsmath}  
\usepackage{comment}
\usepackage{enumerate}
\usepackage{graphicx}
\usepackage{color}
\usepackage{ifpdf}
\usepackage{mdwlist}
\newtheorem{theorem}{Theorem}
\newtheorem{definition}{Definition}
\newtheorem{Remark}{Remark}
\newtheorem{lemma}{Lemma}
\newtheorem{prop}{Proposition}

\def\A{\mathscr{A}}
\def\B{\mathscr{B}}

\def\U{\mathcal{U}}

\def\a{c}
\def\L{Lip(\varphi)}

\title[Weak KAM theorem]{Weak KAM theorem for Hamilton-Jacobi equations}

\author[X. Su]{Xifeng Su}
\address{School of Mathematics\\
Beijing Normal University\\
No. 19, XinJieKouWai St.,HaiDian District\\
 Beijing 100875, P. R. China}
\email{billy3492@gmail.com, xfsu@bnu.edu.cn}

\author[J. Yan]{Jun Yan}
\address{School of Mathematical Sciences and Key Lab. of Mathematics
for Nonlinear Science, Fudan University, Shanghai 200433, China}
\email{yanjun@fudan.edu.cn}

\begin{document}
\maketitle

\begin{abstract}
In this paper, we generalize weak KAM theorem from positive Lagrangian systems to ``proper'' Hamilton-Jacobi equations.

We introduce an implicitly defined solution semigroup of evolutionary Hamilton-Jacobi equations.
By exploring the properties of the solution semigroup, we prove the convergence of solution semigroup and
existence of weak KAM solutions for stationary equations:
\begin{equation*}
H(x, u, d_x u)=0.
\end{equation*}
\end{abstract}




\section{Introduction}
Let $M$ be a compact and connected smooth manifold of dimension $m$
without boundary. We denote by $TM$ and $T^*M$ the tangent bundle
and the cotangent bundle respectively. We denote $\pi:TM\rightarrow
M$ the canonical projection. A point in $TM$ will be denoted by
$(x,v)$ with $x\in M$ and $v\in T_x M=\pi^{-1}(x)$. Likewise, a
point in $T^*M$ will be denoted by $x\in M$ and $p\in T_x^* M$, a
linear form on the vector space $T_x M$. We will fix a Riemannian
metric $g$ on $M$ once and for all. For $v\in T_x M$, the norm
$\|v\|_x$ is $g(v,v)^{\frac{1}{2}}$ and we will also denote by
$\|\cdot\|_x$ the dual norm on $T_x^*M$ if it does not cause
confusion.

We suppose $H:T^*M\times \mathbb{R}\rightarrow \mathbb{R}$ is a
$C^\infty$ function satisfying the following conditions:
\begin{itemize}
\item [\bf{(H1)}] For each $(x,p,u)\in T^*M\times \mathbb{R}$, the Hessian matrix $\frac{\partial^2 H}{\partial
p^2}(x,u,p)$ is everywhere positive definite.

\item [\bf{(H2)}] For each $u\in\mathbb{R}$, $\lim_{\|p\|_x\rightarrow\infty} \frac{H(x,u,p)}{\|p\|_x} =
+\infty$ uniformly in $x\in M$, where $\|\cdot\|_x$ denotes the norm
on $T^*_x M$ induced by a Riemannian Metric.

\item [\bf{(H3)}] The flow $\Phi^t$ of \eqref{characteristic equ} is
complete. That is, the maximal solution of \eqref{characteristic
equ} are defined on all of $\mathbb{R}$.

\item [\bf{(H4)}] $H(x,u,p)$ is uniformly Lipschitz continuous with respect to $u$. We will denote by $\lambda\geq0$ the
corresponding Lipschitz constant
\[
\lambda = \sup_{\substack{u_1,u_2\in \mathbb{R}\\ u_1\neq u_2} }
\frac{|H(x,u_1,p) - H(x,u_2,p)|}{|u_1 - u_2|}  \qquad \forall
~(x,p)\in T^*M.
\]

\item [\bf{(H5)}]$H(x, u, p )$ is increasing with respect to $u$.

\item [\bf{(H6)}]There exists a real number $\a$ such that
\[
c(L(x, \a, \dot{x})) = 0
\]where
\[
c(L(x,\a,\dot{x})) = \inf_{u\in C^{1,1}(M)}
\max_{x\in M} H(x, \a, d_x u(x))
\]is referred as the Ma\~n\'e critical value.
\end{itemize}

We would remark that the completeness assumption of the phase flow $\Phi^t$ is
to exclude the case that the Tonelli minimizers are not $C^1$.
See \cite{Ball'85} for a counterexample in the case of time periodic positive definite Lagrangian systems.

We also point out that (H5) is crucial for the convergence of the solution semigroup in Section \ref{convergence section},
which is referred as \emph{proper condition} (see \cite{Lions'92}).

One example satisfying (H1)-(H5) to keep in mind could be
\begin{equation}
H(x,u,d_x u) = u + H_1(x,d_x u)
\end{equation}where $H_1$ is the usual Tonelli Hamiltonian.

The corresponding evolutionary first-order Hamilton-Jacobi equation
is:
\begin{equation}\label{Cauchy problem}
\left\{\!\!\!
  \begin{array}{rl}
   &\frac{\partial u}{\partial t} + H(x, u, d_x u)=0  \qquad \text{on }M\times \mathbb{R}\\
   &u(x,0)=\varphi(x)\qquad \text{on } M;
  \end{array}
\right.
\end{equation}where $\varphi$ is a given continuous function on $M$.

This first-order nonlinear PDE is of the most general form in the
sense that the unknown function enters explicitly.
The fact that the Hamiltonian depends on the unknown function will be the main obstacle for us in this paper.
We will overcome this difficulty by introducing new approaches such as
defining implicitly the solution semigroup.

Our approach is based on both characteristic method and dynamical approach.

On the one hand, it is because PDEs have not been nearly so well studied as ODEs.
Characteristic method can reduce a problem in partial differential
equations to a problem in ordinary differential equations. See
\cite{Arnold'92,Lions'82,Benton'77} for more detailed description
for characteristics method.

On the other hand, dynamical systems have had a period of fast development within the last three decades.
The dynamical approach here is mainly to employ the theory of ODEs, dynamical systems and
variational methods to find ``integrable structure" (weak KAM
solutions) within general first-order nonlinear PDEs.

By characteristic method, it suffices to deal with the characteristic equation of \eqref{Cauchy problem}, an ODE system:
\begin{equation}\label{characteristic equ}
\left\{\!\!\!
  \begin{array}{rl}
  &\dot{x} = \frac{\partial H}{\partial p}(x, u, p)\\
  &\dot{p} = -\frac{\partial H}{\partial x}(x, u, p) - \frac{\partial H}{\partial u}(x, u, p)\ p\\
  &\dot{u} =  \frac{\partial H}{\partial p}(x, u, p) \ p - H(x, u, p).
\end{array}
\right.
\end{equation}
The phase curves
of the above system on the $2m+1$-dimensional space
$T^*M\times\mathbb{R}$ are called the \emph{characteristics} of \eqref{Cauchy problem}.
Moreover, this system defines a time independent vector field $E$ on $T^*M\times\mathbb{R}$
and generates a flow of diffeomorphisms denoted as $\Phi^t$ from
$T^*M\times\mathbb{R}$ to itself.

We will prove the existence and regularity theorem
for the calibrated curves of \eqref{characteristic equ} which minimize the action.
We will furthermore show that such calibrated curves are characteristics.
These results are analogues of the Tonelli's theorem and Weierstrass's theorem in Mather theory.

We begin with a quick recounting of the main results in the
literature of Mather theory and weak KAM theory which are global and
non-perturbative theories. See \cite{Evans'04,Kaloshin'05}.

The classical weak KAM theorem for Hamilton-Jacobi equation by A.
Fathi \cite{Fathi'97a,Fathi'97b} and W. E \cite{E'99} makes a bridge
between the celebrated Mather theory \cite{Mather'91,Mather'93} and
the classical theory of viscosity solution of Hamilton-Jacobi
equation \cite{Lions'81,Lions'83,Lions'84,Lions'82}.

In the present paper, we will introduce the remarkable tool, the
solution semigroup of characteristic equation \eqref{characteristic
equ}, which is an analogue of the Lax-Oleinik operator
\cite{Hopf'50,Lax'57,Oleinik'57} or \cite{Fathi'08} in classical
weak KAM theory. See also \cite{Yan'12} for a new kind of
Lax-Oleinik type operator with parameters associated with time
periodic positive definite Lagrangian Systems.

By the solution semigroup, we will establish the weak KAM theorem, 
the existence of variational solutions for \eqref{Cauchy problem},
which is just a dynamical description of the viscosity solutions in Section \ref{weak KAM type framework}.

Our next goal is to prove the the convergence of the solution semigroup
which asserts the existence of the weak KAM solutions
for the stationary Hamilton-Jacobi equation:
\begin{equation}\label{Dirichlet problem}
H(x, u, d_x u)=0\qquad \text{on } M.
\end{equation}

We will show that the limit points of the solution semigroup acting
on an arbitrary $u\in C(M\times\mathbb{R},\mathbb{R})$ as initial
condition converges to a weak KAM solution of
\eqref{Dirichlet problem}.

\section{Dynamics of the calibrated curves}\label{dynamics section}
In this section, we will first give a proof of a lower semi-continuity property of the action, which is an analogue of Tonelli's theorem.
And then, we will introduce several key concepts, such as solution semigroup, calibrated curves, which will be useful for our dynamical approach.
Indeed, we show the existence of the solution semigroup and its associated calibrated curves in our setting.

In addition, to apply variational approach with a dynamical interpretation,
we will investigate the relation between calibrated curves and characteristics and prove the the regularity of the calibrated curves.

\subsection{The existence of action minimizing curves} For every $t>0$ and $u\in
C^0(M\times\mathbb{R},\mathbb{R})$, we first define the action of an
absolutely continuous curve $\gamma: [0,t]\rightarrow M$ as
\begin{equation}
A_u^t(\gamma) = \int_0^t L\big(\gamma(s),u\big(\gamma(s),s\big),
\dot\gamma(s)\big) \ ds,
\end{equation}where $L(x,u, \dot{x}) = \sup_{p\in T^*_x M } \{\langle p, \dot{x} \rangle - H(x,u,
p)\}$, which will be referred as the Lagrangian.
\begin{Remark}
Note that the Lagragian depends on the parameter in $\mathbb{R}$.
More precisely, along the absolutely continuous curve, the
Lagrangian is pointwisely defined.
\end{Remark}

In this section, we will prove the existence of curves
$\gamma:[0,t]\rightarrow M$ which minimize the action $A_u^t$ over
the class of absolutely continuous curves subject to a fixed
boundary condition. 
We will use the arguments inspired by
\cite{Mather'91} and \cite{Fathi'08}, but the proof here involves
some more complexity.

In what follows,we will omit $t$ when it doesn't cause confusions.
Note that $A_u(\gamma)$ exists since $L$ is bound below although it
may be $+\infty$.

\begin{prop}\label{property of the Lagrangian}
\begin{enumerate}[(1)]
\item Assume that $H(x, u, p)$ has positive definite
fiberwise Hessian second derivative and superlinear growth. Then, for every $u\in\mathbb{R}$, $L(x,u,\dot{x})$ has positive definite
fiberwise Hessian second derivative and superlinear growth.

\item Assume that $H(x,u,p)$ is increasing and uniformly Lipschitz continuous with respect
to $u$ and $\lambda\geq 0$ is the coresponding Lipschitz constant. Then, $L(x,u,\dot{x})$ is decreasing and uniformly Lipschitz continuous with respect
to $u$ and the Lipschitz constant is $\leq \lambda$.
\end{enumerate}
\end{prop}
\begin{proof}
(1) results from the definition of positive definiteness, superlinear growth and the relation between $H$ and $L$.

To show (2), we notice that
\[
L(x,u, \dot{x}) = \sup_{p\in T^*_x M } \{\langle p, \dot{x} \rangle - H(x,u,p)\}.
\]

Suppose that $u_1,u_2 \in \mathbb{R}$ and $u_1\geq u_2$. For any $(x,\dot{x})\in T_x M$ there exists a unique $p_0$ such that
$\dot{x} = \frac{\partial H}{\partial p}(x, u, p_0)$ and $L(x,u, \dot{x}) = \langle p_0, \dot{x} \rangle - H(x,u,p_0)$.

Therefore,
\begin{equation*}
\begin{split}
L(x,u_1, \dot{x}) &= \sup_{p\in T^*_x M } \{\langle p, \dot{x} \rangle - H(x,u,p)\}\\
                  &= \langle p_0, \dot{x} \rangle - H(x,u,p_0)\\
                  &\leq \langle p_0, \dot{x} \rangle - H(x,u_2,p_0)\\
                  &\leq L(x, u_2, \dot{x}),
\end{split}
\end{equation*}which shows the monotonicity of $L(x,u, \dot{x})$ with respect to $u$.

To prove the Lipschitz property, we notice that
\begin{equation*}
\begin{split}
L(x, u_1, \dot{x}) &= \langle p_0, \dot{x} \rangle - H(x,u_1,p_0)\\
                 &\leq  \langle p_0, \dot{x} \rangle - H(x,u_2,p_0) + \lambda |u_1 - u_2|\\
                 &\leq L(x,u_2, \dot{x}) + \lambda\  |u_1 - u_2|.
\end{split}
\end{equation*}
Likewise, we have $L(x,u_2, \dot{x}) \leq L(x, u_1, \dot{x}) + \lambda\  |u_1 - u_2|$.

Hence, we obtain $|L(x,u_2, \dot{x}) - L(x, u_1, \dot{x})| \leq  \lambda\  |u_1 - u_2|$.
\end{proof}

Let us now introduce the following fundamental theorem which asserts
the compactness of certain subsets of $C^{ac} ([0,t],M)$ and will
play an important role in the next sections.
\begin{theorem}[Tonelli's Theorem]\label{compact lemma}
Assume that (H1)-(H4) hold.
Let $K\in\mathbb{R}$ and $u\in C^0(M\times\mathbb{R},\mathbb{R})$.
The set
\begin{equation*}
C^{ac}_K \equiv \left\{ \gamma\in C^{ac} ([0,t],M) : ~A_u (\gamma)
\leq K\right\}
\end{equation*}is compact in the $C^0$-topology.
\end{theorem}

To keep the pace of the exposition, we postpone
its proof in the appendix of the paper.

\subsection{Solution semigroup}
We will first deduce the solution semigroup theory for the
evolutionary Hamilton-Jacobi equations \eqref{Cauchy problem}. The
key point is that the solution semigroup is defined implicitly.
To the best of our knowledge, similar semigroup were considered in \cite{Douglis'65} for some cases.
See \cite{Benton'77} and references therein for an elementary
introduction to the solution semigroup with more restrictions.

\subsubsection{Well-definition of solution semigroup}
For every given continuous function $\varphi$ on $M$, we now define
the operator $\A:C^0(M\times \mathbb{R}, \mathbb{R}
)\circlearrowleft$ depending on $\varphi$ as follows:
\begin{equation}\label{operator A}
\begin{split}
\A[u](x,t) &= \inf_{\substack{\gamma(t)=x\\ \gamma\in
C^{ac}([0,t],M)}} \left\{ \varphi\big(\gamma(0)\big) + \int_0^t
L\big(\gamma(s),u\big(\gamma(s),s\big), \dot\gamma(s)\big) \ ds
\right\}\\
&= \inf_{\substack{\gamma(t)=x\\ \gamma\in C^{ac}([0,t],M)}} \left\{
\varphi\big(\gamma(0)\big) + A_u(\gamma)\right\},
\end{split}
\end{equation}where $u\in C^0(M\times\mathbb{R}, \mathbb{R})$ and $(x,t)\in
M\times\mathbb{R}^+$.

In the following, we will prove that the operator has a unique fixed point.


\begin{prop}\label{existence of calibrated curves}
For every $x\in M, ~t>0$ and $u\in C^0(M\times
\mathbb{R},\mathbb{R})$, there exists an absolutely continuous curve
$\gamma: [0,t]\rightarrow M$ such that
\begin{equation}
\A[u](x,t) = \varphi\big(\gamma(0)\big) + \int_0^t
L\big(\gamma(s),u\big(\gamma(s),s\big), \dot\gamma(s)\big) \ ds.
\end{equation}
\end{prop}
\begin{proof}
For any $n\in\mathbb{N}$, by the definition of $\A$, there exists a
$\gamma_n\in C^{ac}([0,t],M)$ such that
\begin{equation}
\varphi(\gamma_n(0)) + \int_0^t L(\gamma_n(s), u(\gamma_n(s),s),
\dot{\gamma_n}(s))\  ds \leq \A[u](x,t) + \frac{1}{n}.
\end{equation}

By Theorem \ref{compact lemma}, up to a subsequence, still denoted
by $\gamma_n$, there exists a $\gamma\in C^{ac}([0,t],M)$ such that
$\gamma_n$ converges to $\gamma$ in the $C^0$-topology. Moreover,
\begin{equation}
A_u(\gamma)\leq \liminf_{n\rightarrow+\infty}A_u(\gamma_n).
\end{equation}
Hence $\A[u](x,t) = \varphi(\gamma(0)) + A_u(\gamma)$ which ends the
proof.
\end{proof}

We now study the property of the operator $\A$ defined in
\eqref{operator A}. We claim
\begin{lemma}\label{fixed point existence}
$\A$ has a unique fixed point.
\end{lemma}
\begin{proof}
For any given $t\in\mathbb{R}^+$ and every $u,v\in C^0(M\times \mathbb{R}, \mathbb{R})$, we estimate
\begin{equation*}
\begin{split}
&\big(\A[u] - \A[v]\big)(x,t)\\
\leq& \int_0^t (L\big(\gamma(s),u\big(\gamma(s),s\big),
\dot\gamma(s)\big) - L\big(\gamma(s),v\big(\gamma(s),s\big),
\dot\gamma(s)\big) )\
ds \\
\leq& \lambda \ \|u-v\|_{\infty} t
\end{split}
\end{equation*}where $\gamma\in C^{ac}([0,t],M)$ such that
\[
\A[v](x,t) = \varphi\big(\gamma(0)\big) + \int_0^t
L\big(\gamma(s),v\big(\gamma(s),s\big), \dot\gamma(s)\big) \ ds.
\]

Note that we use here the fact guaranteed by Proposition
\ref{existence of calibrated curves} that the infimum in the
definition of $\A$ is a minimum.

By exchange the position of $u$ and $v$, we obtain
\begin{equation*}
|\big(\A[u] - \A[v]\big)(x,t)| \leq \lambda \ \|u-v\|_{\infty}t.
\end{equation*}

Therefore, we have the following estimates:
\begin{equation*}
\begin{split}
 &\left|\big(\A^2[u] - \A^2[v]\big)(x,t)\right| \\
 \leq & \left| \int_0^t \lambda \big[ \A[u]\big(\gamma(s),s\big) -
 \A[v]\big(\gamma(s),s\big)\big] \ ds\right|\\
 \leq & \int_0^t  s \lambda^2 \|u-v\|_{\infty} \ ds \leq \frac{(t\lambda)^2}{2}
 \|u-v\|_{\infty}.
\end{split}
\end{equation*}

More general, continuing the above procedure, we obtain
\begin{equation}
\left|\big(\A^n[u] - \A^n[v]\big)(x,t)\right| \leq
\frac{(t\lambda)^n}{n!} \|u-v\|_{\infty}.
\end{equation}
Therefore, for any $t\in\mathbb{R}^+$, there exists some $N$ large
enough such that $\A^{N}:~ C^0(M,\mathbb{R})\circlearrowleft$ is a
contraction mapping and has a fixed point $u$. That is, for any
$t\in\mathbb{R}^+$ and $N\in\mathbb{N}$ large enough, there exists a
$u\in C^0(M, \mathbb{R})$ such that
\begin{equation}
\A^N[u](x) = u(x).
\end{equation}

We now show that $u$ is a fixed point of $\A$. Since
\begin{equation*}
\A[u] = \A\circ \A^N[u] = \A^N\circ \A[u],
\end{equation*}$\A[u]$ is also a fixed point of $\A^N$.
By the uniqueness of fixed point of contraction mapping, we have
\[A[u] = u.\]
\end{proof}

We denote $ T_t \varphi(x) = u(x,t)$ the unique fixed point of $\A$, i.e.,
\begin{equation}\label{solution semigroup}
T_t \varphi(x) = \inf_{\substack{\gamma(t)=x\\
\gamma\in C^{ac}([0,t],M)}} \left\{ \varphi\big(\gamma(0)\big) +
\int_0^t L\big(\gamma(s),T_s \varphi\big(\gamma(s)\big),
\dot\gamma(s)\big) \ ds \right\}.
\end{equation}
In the context, $\{T_t\}_{t\geq0}$ will be referred as the solution semigroup.
In the following section, we will show that the family of operators $\{T_t\}_{t\geq 0}$ is a semigroup of nonlinear operators.

\subsubsection{The semigroup property}
\begin{lemma}[Semigroup Property]\label{semi-group property}
$\{T_t\}_{t\geq 0}$ is a one-parameter semigroup of operators from
$C^0(M,\mathbb{R})$ into itself.
\end{lemma}
\begin{proof}
It is easy to see $T_0  = Id$. It suffices to prove that
$T_{t+s}=T_t\circ T_s$ for any $t,s\geq 0$.

For every $\eta\in C^0(M, \mathbb{R})$ and $u \in C^0(M\times
\mathbb{R},\mathbb{R})$, we define
\begin{equation}
\A_t^{\eta}[u](x,t) = \inf_{\substack{\gamma(t)=x\\ \gamma\in
C^{ac}([0,t],M)}} \left\{ \eta\big(\gamma(0)\big) + \int_0^t
L\big(\gamma(s),u\big(\gamma(s),s\big), \dot\gamma(s)\big) \ ds
\right\}.
\end{equation}
By the definition of $T_t$, we have
\begin{equation*}
\begin{split}
T_t\circ T_s \varphi(x) &= \inf_{\substack{{\gamma}(t)=x\\
{\gamma}\in C^{ac}([0,t],M)}} \left\{
T_s\varphi\big({\gamma}(0)\big) + \int_0^{t}
L\big({\gamma}(\tau),T_{\tau}\circ T_s
\varphi\big({\gamma}(\tau)\big), \dot{{\gamma}}(\tau)\big) \ d\tau
\right\}\\
& = \A_t^{T_s \varphi} [T_t\circ T_s \varphi](x).
\end{split}
\end{equation*}

On the other hand,
\begin{equation*}
\begin{split}
T_{t+s} \varphi(x) &= \inf_{\substack{\gamma(t+s)=x\\ \gamma\in
C^{ac}([0,t+s],M)}} \left\{ \varphi\big(\gamma(0)\big) +
\int_0^{t+s} L\big(\gamma(\tau),T_\tau
\varphi\big(\gamma(\tau)\big),
\dot\gamma(\tau)\big) \ d\tau \right\}\\
& = \inf_{\substack{\gamma(t+s)=x\\ \gamma\in C^{ac}([0,t+s],M)}}
\left\{ \varphi\big(\gamma(0)\big) + \bigg(\int_0^{s} +
 \int_{s}^{t+s}\bigg) L\big(\gamma(\tau),T_\tau
\varphi\big(\gamma(\tau)\big), \dot\gamma(\tau)\big) \ d\tau \right\}\\
& = \inf_{\substack{\gamma(t+s)=x\\ \gamma\in C^{ac}([s,t+s],M)}}
\left\{ T_s\varphi\big(\gamma(s)\big) + \int_s^{t+s}
L\big(\gamma(\tau),T_\tau \varphi\big(\gamma(\tau)\big),
\dot\gamma(\tau)\big) \ d\tau \right\}\\
& = \inf_{\substack{\bar{\gamma}(t)=x\\ \bar{\gamma}\in
C^{ac}([0,t],M)}} \left\{ T_s\varphi\big(\bar{\gamma}(0)\big) +
\int_0^{t} L\big(\bar{\gamma}(\tau),T_{\tau+s}
\varphi\big(\bar{\gamma}(\tau)\big),
\dot{\bar{\gamma}}(\tau)\big) \ d\tau \right\}  \\
& = \A_t^{T_s \varphi} [T_{t+s} \varphi](x).
\end{split}
\end{equation*}
By Lemma \ref{fixed point existence}, we know $\A_t^{T_s\varphi}$
has a unique fixed point, i.e. $T_{t+s}=T_t\circ T_s$. This completes
the proof of the Lemma \ref{semi-group property}.
\end{proof}

\subsection{The calibrated curve and its regularity}It is convenient to introduce
the following notion of calibrated curves of the solution semigroup. Our goal in this
section is to conclude the regularity of such calibrated curves.

\begin{definition}
Let $t>0$. We say that a
continuous curve $(\gamma, u): [0,t]\rightarrow M\times \mathbb{R}$
is a calibrated curve for \eqref{Cauchy problem}, if $\gamma\in C^{ac}([0,t],M)$ with
$\gamma(t)=x$ and for any $s\in [0,t]$ we have the
following equality
\begin{equation}
\begin{split}
u(s)  & = \varphi(\gamma(0)) + \int_0^s
L\big(\gamma(\tau),u(\tau), \dot\gamma(\tau)\big) \ d\tau\\
& = T_s \varphi(\gamma(s)).
\end{split}
\end{equation}
\end{definition}

In the next, we will consist in showing the following regularity result.
\begin{theorem} \label{regularity result}
Assume that (H1)-(H5) hold.
The calibrated curve
$(\gamma, u)$ for \eqref{Cauchy problem} is $C^1$ and the curve
\[ \big(\gamma(s),u(s),p(s)=\frac{\partial
L}{\partial \dot{x}}(\gamma(s),u(s),\dot{\gamma}(s))\big) \qquad
\forall ~s\in [0,t]
\]still referred as calibrated curves for convenience if there is no confusion,
satisfies the characteristic equation \eqref{characteristic equ}.
\end{theorem}

\begin{proof}
Suppose that $(\gamma(s), u(s))$ is a calibrated curve for \eqref{Cauchy problem} with $\gamma(t) =x$.
From the definition of the calibrated curves and Theorem \ref{compact lemma},
we know that $\gamma(s)$ is
absolutely continuous for any $s\in [0,t]$ and so it is
differentiable almost everywhere.

Let us start by fixing $t_0\in[0,t]$ where $\gamma$ is
differentiable. Let $k= \|\dot{\gamma}(t_0)\|$.

To fix notation, we denote
\begin{equation*}
x_0 = \gamma(t_0), v_0 = \dot{\gamma}(t_0), u_0 = T_{t_0}\varphi(x_0), p_0 = \frac{\partial L}{\partial v}(x_0, u_0, v_0).
\end{equation*}

Denote $\Phi^t:T^*M\times\mathbb{R}\circlearrowleft$ the flow of
diffeomorphisms generated by the characteristic equation
\eqref{characteristic equ}. Let
\begin{equation*}
\begin{split}
\overline{B}_{2k} &= \left\{ v\in T_{x_0} M ~:~ \|v\| \leq 2k \right\}, \\
B_{2k} &=  \mathcal{L}\overline{B}_{2k},
\end{split}
\end{equation*}where $\mathcal{L}$ is the Legendre transform associated with $L$.

When $\epsilon$ is sufficiently small, by the characteristic method, it is easy to check the following facts.
\begin{prop}\label{basic facts}
\begin{itemize}
\item [(1)] For every $0<t<\epsilon$, $\Phi^t$ considered as a mapping from $\{x_0\}\times \{u_0\}\times B_{2k}$
to its image is a diffeomorphism.

\item [(2)] Let $\Phi: (0,\epsilon) \times \{x_0\}\times \{u_0\}\times B_{2k} \rightarrow \mathbb{R}\times \mathbb{R}\times T^* M$.
Denote $\overline{\Omega}_{2k}^\epsilon$, a subset of $\mathbb{R}\times \mathbb{R}\times T^* M$, the image of $\Phi$.
Let \[\Omega^{\epsilon}_{2k} = \pi (\overline{\Omega}_{2k}^\epsilon)\]
where $\pi: \mathbb{R}\times \mathbb{R}\times T^*M\rightarrow \mathbb{R}\times M$,
$(t,u,x, p) \mapsto (t,x)$.\newline
Then, $\pi\circ \Phi$ is a diffeomorphism.

\item [(3)] For simplicity, we let
\begin{equation*}
\Gamma(s) = \gamma(s+t_0),
v(s) = u(s+t_0),
\end{equation*} where $\Gamma(0) = \gamma(t_0) = x_0,v(0) = u(t_0) = u_0$.\newline
Then, the graph of $\Gamma$ on $[0,\epsilon]$ is included in $\Omega^{\epsilon}_{2k}$.
\end{itemize}
\end{prop}

Now we claim that there exists a
$\epsilon >0$ such that $\gamma(s)$ is differentiable on
$[t_0,t_0+\epsilon]$ and the curve
\[
\big(x(s)=\gamma(s),u(s),p(s)=\frac{\partial L}{\partial
\dot{x}}(\gamma(s),u(s),\dot{\gamma}(s))\big) \text{ with }s\in
[t_0,t_0+\epsilon]
\]generated by the calibrated curve
$(\gamma,u)$ satisfies the characteristic equation~\eqref{characteristic equ}.
We will divide our proof into two steps.

\textbf{Step 1.}
We will first construct a classical solution of \eqref{Cauchy problem} in $\Omega^\epsilon_{2k}$.

We denote the rectangle in $M\times\mathbb{R}$ by
\[
I_{\epsilon,\tau}(x,s) = \{(y,t)\in M\times
\mathbb{R}~:~dist(y,x)\leq \epsilon,~|t-s|\leq \tau\},
\]where $dist$ denotes the distance on $M$ associated with the Riemannian metric $g$.
In particular, we denote $I_{\epsilon}(x) = \{y\in M~:~dist(y,x)\leq
\epsilon\}$ the rectangle in $M$.

Later on, we will always choose $\epsilon_0,\tau_0>0$ such that
$\Omega^\epsilon_{2k} \subseteq I_{\epsilon_0,\tau_0}(x_0,t_0)$ is included in
some local coordinate chart $U\times\mathbb{R}$ of
$M\times\mathbb{R}$. We have the following fundamental lemma.
\begin{lemma}\label{representation for H-J eq}
Let $\overline{\varphi}_p(x) = u_0 + p\cdot(x-x_0)$. For every $(x,s)\in \Omega^\epsilon_{2k}$, there exist
a $p_0=p_0(x,s)\in B_{2k}$ such that
the Cauchy problem for the Hamilton-Jacobi equation
\begin{equation}\label{Hamilton-Jacobi method}
\left\{\!\!\!
  \begin{array}{rl}
   &\frac{\partial S}{\partial s} + H(x, S, d_x S)=0,\\
   &S(x,t_0)=\overline{\varphi}_{p_0}(x);
  \end{array}
\right.
\end{equation}
has a classical solution $S(x,s)$ in $\Omega^\epsilon_{2k}$. In particular, we have
\begin{equation}
S(x,s)= u_0 + \int_{t_0}^s L\big(x(\tau), S(x(\tau),\tau), \dot{x}(\tau)\big)\ d\tau.
\end{equation}
\end{lemma}
\begin{proof}
We will construct the specific solution below. From the fundamental
existence and uniqueness theorem of ordinary differential equations,
we can take $\tau_0$ small enough such that for each initial value
$\xi_0 \in T^*M \times \mathbb{R}$, there is a characteristic curve $\xi$ satisfying
\begin{equation}
\left\{\!\!\!
  \begin{array}{rl}
   &\dot{\xi}(t) = E(\xi(t))\\
   &\xi(0)=\xi_0,
  \end{array}
\right.
\end{equation}where
\[
E(x,p,u)=\left(\frac{\partial H}{\partial p},-\frac{\partial H}{\partial x} - \frac{\partial H}{\partial u}\ p,
\frac{\partial H}{\partial p} \ p - H\right)
\] is the vector field of \eqref{characteristic equ}.

Due to Proposition \ref{basic facts},  for any $(x,s)\in \Omega^\epsilon_{2k}$,
there exists a unique $p_0=p_0(x,t) \in B_{2k}$ such that for
$t_0\leq t\leq s$  we have the characteristic curve
\begin{equation}
\left\{\!\!\!
  \begin{array}{rl}
  &\Phi^t\big(x_0,p_0,u_0\big) =
  (x(t),p(t),u(t)),\\
  &x(s)=x.
\end{array}
\right.
\end{equation}

Let $S(x,s) = u(s)$, which is a classical solution of
\eqref{Hamilton-Jacobi method} in $\Omega^\epsilon_{2k}$ when $\epsilon$ is sufficiently small.

Hence, we can write
\begin{equation*}
S(x,s)= u_{0} + \int_{t_0}^s L\big(x(\tau), S(x(\tau),\tau), \dot{x}(\tau) \big)\ d\tau,
\end{equation*}which completes the proof of the lemma.
\end{proof}

Before going into the second step, we would point out another fact
which asserts the variational property of the classical solution $S(x,s)$.

Let $\xi:[t_0, s]\rightarrow I_{\epsilon}(x_0) \subseteq M$
with $\xi(t_0)=x_0, \xi(s)=x$ be an absolutely
continuous curve. Since $S(x,s)$ is $C^1$ and the
map $[t_0,s]\rightarrow \mathbb{R}$, $\tau\mapsto
S(x(\tau),\tau)$ is absolutely continuous and thus
we have
\[
S(\xi(s), s) -
S(\xi(t_0), t_0) = \int_{t_0}^{s}\left\{
\frac{\partial S}{\partial s}(\xi(\tau),\tau) +
\frac{\partial S}{\partial x}(\xi(\tau),\tau)
\dot{\xi}(\tau)\right\}\ d\tau.
\]

For each $\tau$ where $\dot{\xi}(\tau)$ exists, the Fenchel inequality
implies
\begin{equation*}
\frac{\partial S}{\partial x}(\xi(\tau),\tau)
\dot{\xi}(\tau) \leq H\big(\xi(\tau),
S(\xi(\tau),\tau)),\frac{\partial
S}{\partial x}(\xi(\tau),\tau)\big) + L(\xi(\tau),
S(\xi(\tau),\tau)), \dot{\xi}(\tau)).
\end{equation*}

Since $S$ satisfies \eqref{Hamilton-Jacobi
method}, we have
\begin{equation}\label{Classical Hamilton-Jacobi method}
S(x,s) - S(x_0,t_0)
\leq \int_{t_0}^{s}
L(\xi(\tau),S(\xi(\tau),\tau), \dot{\xi}(\tau))\  d\tau.
\end{equation}

It is not difficult to check that there exists a $C^1$ curve $\xi$
such that \eqref{Classical Hamilton-Jacobi method} is an equality if
and only if
\[
\dot{\xi}(\tau) = \frac{\partial H}{\partial p}\bigg(x,
S(\xi(\tau),\tau),
\partial_x S(\xi(\tau),\tau)\bigg).
\]

Therefore, we can write the function $S(x,s)$ as
\[
S(x,s) = u_0 +
\inf_{\substack{\xi(s)=x,\xi(t_0)=x_0\\
\xi\in C^{ac}([t_0,s],I_{\epsilon}(x_0))}}
\int_{t_0}^{s} L(\xi(\tau),S(\xi(\tau),\tau),
\dot{\xi}(\tau))\  d\tau.
\]

\textbf{Step 2.} We will show that the classical solution $S(x,t)$ constructed above
is the same as $u(x,t)$ at the point $(\gamma(s),s)$ for any $s\in (t_0,t_0 + \epsilon)$.

We first introduce the following lemma.
\begin{lemma}\label{action minimizing inside the domain}
Let $\epsilon$ be sufficient small. For any $(x,s)\in
\Omega^\epsilon_{2k}$, if $(\gamma,u)$ is a calibrated curve for
\eqref{Cauchy problem} such that
\[\gamma(t_0)=x_0,~u(t_0)=u_0,~\gamma(s)=x,\] then we
have
\[
dist(\gamma(\tau), x_0)<\epsilon_0\qquad \forall ~\tau\in [t_0,s].
\]
\end{lemma}
\begin{proof}
We suppose by contradiction that $\gamma([t_0, s])\nsubseteq
\mathring{I}_{\epsilon_0}(x_0)$ which denotes the inner points of the
rectangle $I_{\epsilon_0}(x_0)$, i.e. there exists $t_0<t_1<s$
such that
\[
\gamma([t_0,t_1))\subseteq \mathring{I}_\epsilon(x_0) \text{ and }
dist(\gamma(t_1),x_0) = \delta.
\]

By the fiberwise superlinear growth of $L$ and Lipschitz continuity
of $L$ with respect to $u$, there exists $C_1>-\infty$ such that for
every $\dot{x}\in T_x M$, we have
\begin{equation}\label{estimate the Lagrangian}
L(x, u, \dot{x}) \geq L(x, 0, \dot{x})- \lambda |u|\geq \|\dot{x}\|
+C_1 -\lambda |u|.
\end{equation}

Since $u$ is continuous, there exists a $K>0$ such that $|u(\tau)|\leq K$
for every $\tau\in [t_0, s]$. Consequently, we obtain from above
inequality \eqref{estimate the Lagrangian} that
\begin{equation}\label{estimate A_u}
A_u(\gamma) = A_u(\gamma|_{[t_0,t_1]}) +
A_u(\gamma|_{[t_1,s]}) \geq \delta + (C_1- \lambda K
)(s-t_0).
\end{equation}

Due to the definition of $(\gamma, u)$, we have
\[
A_u(\gamma)= u(s) - u (t_0).
\]

Hence, letting $s\rightarrow t_0$ in \eqref{estimate A_u}, by
the continuity of $u$, we obtain $0\geq \delta$, which is a
contradiction.
\end{proof}

We continue now with the proof of the second step.

We know that $(\gamma,u)$ is a calibrated curve of  \eqref{Cauchy problem} for $s\in (t_0,t_0 + \epsilon)$ and so we have by the semigroup property of $\{T_t\}_{t\geq 0}$ that
\begin{equation}
\begin{split}
u(s) & = u_0 + \inf_{\substack{\xi(s)=\gamma(s),\xi(t_0)=x_0\\
\gamma\in C^{ac}([t_0,s],I_{\epsilon_0}(x_0))}} \int_{t_0}^s L(\xi(\tau), u(\tau), \dot{\xi}(\tau))\ d\tau\\
& = u_0 + \int_{t_0}^s L(\gamma(\tau), v(\tau), \dot{\gamma}(\tau))\ d\tau.
\end{split}
\end{equation}

Let us denote $\Psi(s) =S(\gamma(s),s) - u(s)$.
We will first claim that along the curve $\gamma$ the quantity $\Psi(s)\leq 0$ for any $s\in (t_0,t_0+ \epsilon)$.

Using a contradiction argument, we assume that there exists an $s_0>t_0$ such that $\Psi(s_0) >0 $. Namely,
\begin{equation}\label{difference of the barrier and S}
\Psi(s_0) \leq \Psi(s) + \int_{s}^{s_0}\left[L(\gamma(\tau), S(\gamma(\tau),\tau), \dot{\gamma}(\tau)) - L(\gamma(\tau), u(\tau), \dot{\gamma}(\tau)) \right]\ d\tau.
\end{equation}

Hence, by the continuity of $\Psi$, one can define
\[
s_1 = \inf \{ s\in [t_0,s_0]~:~ \Psi|_{[s,s_0]}> 0\}.
\]

Clearly $\Psi(s_1) = 0$. Take $s=s_1$ in \eqref{difference of the barrier and S}, we have
\begin{equation*}
\begin{split}
\Psi(s_0) &\leq \int_{s_1}^{s_0}\left[L(\gamma(\tau), S(\gamma(\tau),\tau), \dot{\gamma}(\tau)) - L(\gamma(\tau), u(\tau), \dot{\gamma}(\tau)) \right]\ d\tau\\
&\leq 0.
\end{split}
\end{equation*}The last inequality holds because of the proper condition (H5).
It contradicts the assumption that $\Psi(s_0)>0$ and concludes the claim.

Likewise, along the characteristic curve $x$, the quantity $\Psi(s) \geq 0$ for any $s\in (t_0,t_0+ \epsilon)$.
Thus, we have $S(\gamma(s),s) = u(s)$ which shows this step.

Consequently, due to the arbitrariness of $s$, we conclude that
there exists an $\epsilon>0$ such that $\gamma$ is differentiable on
$[t_0, t_0 + \epsilon]$ and the curve
\[
\big(x(s)=\gamma(s),u(s),p(s)=\frac{\partial L}{\partial
\dot{x}}(\gamma(s),u(s),\dot{\gamma}(s))\big) \text{ with }s\in
[t_0,t_0+\epsilon]
\]satisfies the characteristic equation \eqref{characteristic equ}.

We now extend such local results to the global ones using a similar argument in \cite{Mather'91} or \cite{Fathi'08}.
Suppose that there exists a $t_1\in (t_0, t)$ such that the curve generated by the calibrated curve $(\gamma,u)$
coincides with a characteristic in $[t_0, t_1)$ and $[t_0,t_1)$ is the maximal interval
on which this curve coincides with a characteristic.
Since $(\gamma, u)([t_0, t_1))$ is contained in a compact set $\gamma([0,t])\times u([0,t])$, by (H3),
the characteristic curve can be extended to the compact closure $[t_0,t_1]$.
Therefore, $\dot\gamma(t_1)$ exists. We apply the above argument (Step 1 and Step 2) and obtain
that $\gamma$ is differentiable on $[t_1, t_1 + \epsilon]$ and the curve
\[
\big(x(s)=\gamma(s),u(s),p(s)=\frac{\partial L}{\partial
\dot{x}}(\gamma(s),u(s),\dot{\gamma}(s))\big) \text{ with }s\in
[t_1,t_1+\epsilon]
\]satisfies the characteristic equation \eqref{characteristic equ}.
Because the characteristic curve is unique in the neighborhood of $t_1$, the curve
generated by the calibrated curve $(\gamma,u)$
coincides with a characteristic in $[t_0, t_1+\epsilon)$, which is a contradiction with the maximality of $t_1$.
Hence, $t_1 = t$.

Consequently, it is easy to see that the calibrated curve is differentiable on $[0, t]$ and satisfies the
characteristic equation, which completes the proof of the theorem.
\end{proof}

\section{Weak KAM type framework for Hamilton-Jacobi equations}\label{weak KAM type framework}

\subsection{Existence of variational solutions for \eqref{Cauchy problem}}
Following \cite{Fathi'97a, Contreras'preprint, Yan'12}, we give
the analogous definition of the variational solutions for \eqref{Cauchy problem} ( or weak KAM solutions for \eqref{Dirichlet problem})  with a
dynamical meaning in our setting as follows.

\begin{definition}[Variational Solutions]
We say that $U:M\times\mathbb{R}\rightarrow \mathbb{R}$ is a
 variational solution of \eqref{Cauchy problem} if the
following are satisfied:
\begin{enumerate}[(1)]
\item For any $(x,t_1),(y,t_2)\in M\times\mathbb{R}$ with $0\leq t_1<t_2$,
we have
\begin{equation*}
U(y,t_2) - U(x,t_1) \leq \inf_{\substack{\gamma(t_1)=x,\gamma(t_2)=y\\
\gamma\in C^{ac}([t_1,t_2],M)}}\int_{t_1}^{t_2} L(\gamma(s),
U(\gamma(s),s), \dot{\gamma}(s))ds.
\end{equation*}

\item For any $(x,t)\in M\times\mathbb{R}$, there exists a $C^1$ curve $\gamma:[0,t]$ with $\gamma(t)=x$ such
that
\begin{equation*}
\begin{split}
U(x,t) - U(\gamma(s), s) &= \int_{s}^{t} L(\gamma(\tau),
U(\gamma(\tau),\tau), \dot{\gamma}(\tau))d\tau \\
&=\inf_{\substack{\xi(s)=\gamma(s),\xi(t)=x\\
\xi\in C^{ac}([s,t],M)}}\int_{s}^{t} L(\xi(\tau),
U(\xi(\tau),\tau), \dot{\xi}(\tau))d\tau \quad \forall~ 0\leq s<t.
\end{split}
\end{equation*}
\end{enumerate}

In particular, a variational solution of the stationary equation \eqref{Dirichlet problem} is also called a weak KAM solution. 
\end{definition}

We now claim
\begin{theorem}\label{solution semigroup generate weak KAM solution}
Assume that (H1)-(H5) hold. Then, the solution semigroup $\{T_t\}_{t\geq 0}$ we obtain in \eqref{solution semigroup}
acting on the initial value $\varphi(x)$ is a variational
solution of \eqref{Cauchy problem}.
\end{theorem}
\begin{proof}
For any $(x,t_1),(y,t_2)\in M\times\mathbb{R}$ with $0\leq t_1<t_2$, by
the definition of the solution semigroup, we have
\begin{equation*}
\begin{split}
T_{t_2} \varphi(y) - T_{t_1} \varphi(x) &\leq \inf_{\substack{\gamma(t_2)=y, \gamma(t_1)=x\\
\gamma\in C^{ac}([0,t_2],M)}} \left\{ \varphi\big(\gamma(0)\big) +
\int_0^{t_2} L\big(\gamma(s),T_s \varphi\big(\gamma(s)\big),
\dot\gamma(s)\big) \ ds \right\}\\
&\quad - \inf_{\substack{\xi(t_1)=x\\
\xi\in C^{ac}([0,t_1],M)}} \left\{ \varphi\big(\xi(0)\big) +
\int_0^{t_1} L\big(\xi(s),T_s \varphi\big(\xi(s)\big),
\dot\xi(s)\big) \ ds \right\} \\
& \leq \inf_{\substack{\gamma(t_2)=y, \gamma(t_1)=x\\
\gamma\in C^{ac}([t_1,t_2],M)}} \left\{ \int_{t_1}^{t_2}
L\big(\gamma(s),T_s \varphi\big(\gamma(s)\big), \dot\gamma(s)\big) \
ds \right\}.
\end{split}
\end{equation*}
The last inequality holds because one can choose $\gamma\in
C^{ac}([0,t_2],M)$ such that $\gamma=\xi$ on the interval $[0,t_1]$
where $\xi\in C^{ac}([0,t_1],M)$ with $\xi(t_1)=x$ satisfies
\[
T_{t_1}\varphi(x) = \left\{ \varphi\big(\xi(0)\big) + \int_0^{t_1}
L\big(\xi(s),T_s \varphi\big(\xi(s)\big), \dot\xi(s)\big) \ ds
\right\}.
\]

All what remains is to show (2). For any $(x,t)\in M\times
\mathbb{R}^+$, there exists a minimizing curve $\gamma_t\in
C^{1}([0,t], M)$ with $\gamma_t(t) = x$ such that
\begin{equation*}
T_t \varphi(x) =  \varphi\big(\gamma_t(0)\big) + \int_0^t
L\big(\gamma_t(s),T_s \varphi\big(\gamma_t(s)\big),
\dot\gamma_t(s)\big) \ ds.
\end{equation*}
It follows from the proof of Lemma \ref{semi-group property} that
\begin{equation*}
T_t \varphi(x) - T_{s} \varphi\big(\gamma_t(s)\big) =
\int_s^t L\big(\gamma_t(\tau),T_{\tau}
\varphi\big(\gamma_t(\tau)\big),
\dot{\gamma}_t(\tau)\big) \ d\tau
\end{equation*}for any $s\in[0,t]$. This ends the proof of the theorem.

\end{proof}

\subsection{Relationship between variational solutions and viscosity solutions}
Crandall and Lions \cite{Lions'83} have introduced the following
notion of viscosity solutions which applies naturally to first-order Hamilton-Jacobi equations.
\begin{definition}[Viscosity solution]
A function $U:V\rightarrow \mathbb{R}$ is a viscosity sub-solution
of \eqref{Cauchy problem} on the open subset $V\subseteq M\times
\mathbb{R}$, if for every $C^1$ function $\phi :V\rightarrow
\mathbb{R}$ with $\phi\geq U$ everywhere and
$U(x_0,t_0)=\varphi(x_0,t_0)$ at every point $(x_0,t_0)\in V$, we
have
\[
\frac{\partial \phi}{\partial t}(x_0,t_0) + H(x_0, \phi(x_0,t_0),
d_x \phi(x_0,t_0))\leq 0.
\]

A function $U:V\rightarrow \mathbb{R}$ is a viscosity super-solution
of \eqref{Cauchy problem} on the open subset $V\subseteq M\times
\mathbb{R}$, if for every $C^1$ function $\psi :V\rightarrow
\mathbb{R}$ with $\psi\leq U$ everywhere and
$U(y_0,\tau_0)=\varphi(y_0,\tau_0)$ at every point $(y_0,\tau_0)\in
V$, we have
\[
\frac{\partial \psi}{\partial t}(y_0,\tau_0) + H(y_0,
\psi(y_0,\tau_0), d_x \psi(y_0,\tau_0))\geq 0.
\]

A function $u:V\rightarrow \mathbb{R}$ is a viscosity solution of
\eqref{Cauchy problem} on the open subset $V\subseteq
M\times\mathbb{R}$, if it is both a sub-solution and a
super-solution.
\end{definition}

We are now ready to establish the relationship between variational
solutions and viscosity solutions in our context.

\begin{theorem}
Assume that (H1)-(H5) hold. Then,
any variational solution of \eqref{Cauchy problem} is also a viscosity
solution and vice versa.
\end{theorem}
\begin{proof}
Let $V$ be an open subset of $M\times\mathbb{R}$.

(I). To prove that variational solutions are viscosity solutions, it
suffices to show that $U$ is both a viscosity sub-solution and a
viscosity super-solution.

(a) Let $\phi: V\rightarrow \mathbb{R}$ be $C^1$ such that
$U\leq\phi$ with equality at $(x_0,t_0)\in V$. Therefore,
\[
\phi(x_0,t_0) - \phi(x,t) \leq U(x_0, t_0) - U(x, t).
\]

Take $v\in T_{x_0}M$ and pick
$\gamma:(t_0-\delta,t_0+\delta)\rightarrow M$ a $C^1$ curve with
$\gamma(t_0)=x_0$ and $\dot{\gamma}(t_0)=v$. For
$t\in(t_0-\delta,t_0)$, by the above inequality and the definition
of a variational solution, we obtain
\begin{equation*}
\begin{split}
\phi(\gamma(t_0),t_0) - \phi(\gamma(t),t)&\leq U(\gamma(t_0), t_0) -
U(\gamma(t), t)\\
&\leq \int_{t}^{t_0} L(\gamma(s), U(\gamma(s),s),
\dot{\gamma}(s))ds.
\end{split}
\end{equation*}
Dividing by $t_0-t$ at each side of the above inequality, we have
\begin{equation*}
\frac{\phi(\gamma(t),t) - \phi(\gamma(t_0),t_0)}{t-t_0} \leq
\frac{1}{t_0-t} \int_{t}^{t_0} L(\gamma(s), U(\gamma(s),s),
\dot{\gamma}(s))ds.
\end{equation*}
Letting $t\rightarrow {t_0}_-$ yields
\begin{equation}\label{sub-solution inequality}
\frac{\partial \phi}{\partial t}(x_0,t_0) + d_x\phi(x_0,t_0) v -
L(x_0, \phi(x_0,t_0), v)\leq 0
\end{equation}

Taking supremum over $v\in T_{x_0} M$ for \eqref{sub-solution
inequality}, we obtain
\[
\frac{\partial \phi}{\partial t}(x_0,t_0) + H(x_0, \phi(x_0,t_0),
d_x \phi(x_0,t_0))\leq 0
\]
which shows that $U$ is a viscosity sub-solution.

(b) Suppose that $\psi: V\rightarrow \mathbb{R}$ be $C^1$ such that
$U\geq\phi$ with equality at $(x_0,t_0)\in V$. This implies
\[
\psi(x_0,t_0) - \psi(x,t) \geq U(x_0, t_0) - U(x, t).
\]

By the definition of variational solutions, we can choose a $C^1$ curve
$\gamma:[0, t_0]$ with $\gamma(t_0)=x_0$ such that
\begin{equation*}
U(\gamma(t_0),t_0) - U(\gamma(s), s) = \int_{s}^{t_0}
L(\gamma(\tau), U(\gamma(\tau),\tau), \dot{\gamma}(\tau))d\tau \quad
\forall ~0\leq s<t_0.
\end{equation*}
Therefore,
\begin{equation}\label{super-solution inequality}
\psi(\gamma(t_0),t_0) - \psi(\gamma(s), s) \geq \int_{s}^{t_0}
L(\gamma(\tau), U(\gamma(\tau),\tau), \dot{\gamma}(\tau))d\tau \quad
\forall ~0\leq s<t_0.
\end{equation}
We divide both sides of \eqref{super-solution inequality} by $t_0-s$
and get
\[
\frac{\psi(\gamma(s), s) - \psi(\gamma(t_0),t_0)}{s-t_0} \geq
\frac{1}{t_0-s}\int_{s}^{t_0} L(\gamma(\tau), U(\gamma(\tau),\tau),
\dot{\gamma}(\tau))d\tau \quad \forall s<t_0.
\]

Let $s\rightarrow {t_0}_-$ and we have
\[
\frac{\partial \psi}{\partial t}(x_0,t_0) + d_x\psi(x_0,t_0)
\dot{\gamma}(t_0) - L(x_0, \psi(x_0,t_0), \dot{\gamma}(t_0))\geq 0.
\]

This yields
\[
\frac{\partial \psi}{\partial t}(x_0,t_0) + H(x_0, \psi(x_0,t_0),
d_x \psi(x_0,t_0))\geq 0,
\]which completes the proof of the first part of the theorem.

(II). Notice that, by Theorem \ref{solution semigroup generate weak KAM solution},
$u(x,t) = T_t\varphi(x)$ is a variational solution of \eqref{Cauchy problem}. 
So, to prove that viscosity solutions are variational solutions,
it is enough to show that the viscosity for \eqref{Cauchy problem} is unique.

Before going into the proof of the uniqueness result, we will introduce the following estimate
whose proof is essentially given in \cite{Barles'13}[Section 5.2] and will be omitted here.
\begin{lemma}
Suppose that $H\in C^2$ satisfies (H5). Let $u_1(x,t), u_2(x,t)$ are two viscosity solutions of \eqref{Cauchy problem}.
If either $u_1(x, t)$ or $u_2(x,t)$ is uniformly Lipschitz continuous on $M\times [0,T]$, we have
\begin{equation}
\sup_{M\times [0, T]} (u_1-u_2) \leq \sup_{M} (u_1(x, 0)-u_2(x,0)).
\end{equation}
\end{lemma}

In the sequel, we will use this estimate to obtain the uniqueness of viscosity solution of \eqref{Cauchy problem}.
Let $u_1(x,t) = T_t \varphi(x)$ be the variational solution and so it is a viscosity solution.
Suppose that $u_2(x,t)$ is another viscosity solution of \eqref{Cauchy problem}.

Since, for any given $\delta>0$, $u_1(x,t)$ is uniformly Lipschitz continuous on $M\times [\delta, T]$, we have
\begin{equation*}
\sup_{M\times [\delta, T]} (u_1-u_2) \leq \sup_{M} (u_1(x, \delta)-u_2(x,\delta))
\end{equation*}and
\begin{equation*}
\sup_{M\times [\delta, T]} (u_2-u_1) \leq \sup_{M} (u_2(x, \delta)-u_1(x,\delta)).
\end{equation*}

Due to the arbitrariness of $\delta$, the continuity of $u_1(x,t)$ and $u_2(x,t)$ with respect to $t$ and the initial condition $u_1(x,0)=u_2(x,0)=\varphi(x)$, we obtain
\[
u_1(x,t) = u_2 (x,t) \qquad \forall ~(x,t)\in M\times [0,t],
\]which shows the equivalence relation between variational solutions and viscosity solutions.

\end{proof}

\section{convergence of the solution semigroup}\label{convergence section}
This section is devoted to showing that the solution semigroup $\{T_t\}_{t\geq 0}$
with an arbitrary $\varphi\in C^0(M,\mathbb{R})$ as initial condition
converges to a weak KAM solution of \eqref{Dirichlet problem} as $t\rightarrow +\infty$.

\subsection{Properties of the solution semigroup}
Before going into the details of the proof of our main theorem, we will first obtain the
several crucial properties of $\{T_t\}_{t\geq 0}$ and then 
show the Lipschitz property of variational solutions.

Here are two important properties of $\{T_t\}_{t\geq 0}$.
\begin{lemma}\label{monotonicity and nonexpanding}
Assume that (H1)-(H5) hold. Then,
\begin{itemize}
\item [(1)](Monotonicity)\qquad $\{T_t\}_{t\geq 0}$ is increasing.

\item [(2)](Non-expansiveness) \qquad $\{T_t\}_{t\geq 0}$ is non-expanding.
\end{itemize}
\end{lemma}
\begin{proof}
(1). For given $\varphi,~\psi\in C^0(M,\mathbb{R})$ with
$\varphi\leq \psi$, we suppose, by contradiction, that there exist
$t_1>0$ and $x\in M$ such that $T_{t_1}\varphi(x) > T_{t_1}
\psi(x)$.

By the definition and semi-group property of $\{T_t\}_{t\geq0}$, we
obtain
\begin{equation}\label{monotonicity estimate}
\begin{split}
&\quad T_{t_1}\varphi(x) - T_{t_1}\psi(x) \\&= \inf_{\substack{\gamma(t_1)=x\\
\gamma\in C^{ac}([0,t_1],M)}} \left\{ T_s\varphi\big(\gamma(s)\big)
+ \int_s^{t_1} L\big(\gamma(\tau),T_\tau
\varphi\big(\gamma(\tau)\big),
\dot\gamma(\tau)\big) \ d\tau \right\} \\
&\quad-\inf_{\substack{\gamma(t_1)=x\\ \gamma\in C^{ac}([0,t_1],M)}}
\left\{ T_s\psi\big(\gamma(s)\big) + \int_s^{t_1}
L\big(\gamma(\tau),T_\tau
\psi\big(\gamma(\tau)\big), \dot\gamma(\tau)\big) \ d\tau \right\} \\
& \leq T_s \varphi(\Gamma(s)) - T_s \psi(\Gamma(s)) \ +\\
&\qquad \int_s^{t_1}\left( L\big(\Gamma(\tau),T_\tau
\varphi\big(\Gamma(\tau)\big), \dot\Gamma(\tau)\big) -
L\big(\Gamma(\tau),T_\tau \psi\big(\Gamma(\tau)\big),
\dot\Gamma(\tau)\big)\right) \ d\tau
\end{split}
\end{equation}where $\Gamma\in C^{ac}([0,t_1],M)$ such that
\[
T_{t_1}\psi(x) = \psi(\Gamma(0)) + \int_0^{t_1} L\big(\Gamma(s),T_s
\psi\big(\Gamma(s)\big), \dot\Gamma(s)\big) \ ds.
\]

Let $\Psi(\tau) = T_\tau \varphi(\Gamma(\tau)) - T_\tau
\psi(\Gamma(\tau))$. We can rewrite \eqref{monotonicity estimate} as
\begin{equation*}
\Psi(t_1) \leq \Psi(s) + \int_s^{t_1} \left(
L\big(\Gamma(\tau),T_\tau \varphi\big(\Gamma(\tau)\big),
\dot\Gamma(\tau)\big) - L\big(\Gamma(\tau),T_\tau
\psi\big(\Gamma(\tau)\big), \dot\Gamma(\tau)\big)\right) \ d\tau.
\end{equation*}

Hence, by the continuity of $\Psi$ and the fact that $\Psi(0)\leq 0$
and $\Psi(t_1)>0$, we have that there exists $0\leq t_0<t_1$ such
that $\Psi(t_0) = 0$. We define
\begin{equation}\label{define t2}
t_2 = \inf \{~t\in [t_0,t_1]:~\Psi|_{[t,t_1]} >0 ~\}.
\end{equation}

Clearly $\Psi(t_2) = 0$. Therefore, by the monotonicity of $L$, we
obtain that
\begin{equation*}
0<\Psi(t_1) \leq \Psi(t_2) = 0 \quad \text{ for }s = t_2,
\end{equation*}which is a contradiction. This finishes the proof of
point (1) of the lemma.

We continue to prove (2) in the same spirit as (1). For each
$\varphi, \psi \in C^0(M, \mathbb{R})$, without loss of generality,
we suppose that there exist $t_1>0$ and $x\in M$ such that
$T_{t_1}\varphi(x) > T_{t_1} \psi(x) + \|\varphi - \psi\|_\infty$.

We just apply the same argument for \[\overline\Psi(\tau) = T_\tau
\varphi(\Gamma(\tau)) - T_\tau \psi(\Gamma(\tau)) - \|\varphi -
\psi\|_\infty\] instead of $\Psi(\tau) = T_\tau \varphi(\Gamma(\tau)) - T_\tau
\psi(\Gamma(\tau))$. We define $t_2$ by substituting $\Psi$ with
$\overline\Psi$ in \eqref{define t2} and have
\begin{equation*}
0<\overline{\Psi}(t_1) \leq \overline\Psi(t_2)=0,
\end{equation*}which is contradiction. Therefore, we have
\[
T_t\varphi(x) - T_t\psi(x) \leq \|\varphi - \psi\|_\infty.
\]
Likewise, one can get
\[
T_t\varphi(x) - T_t\psi(x) \geq - \|\varphi - \psi\|_\infty,
\]
which completes the proof of the lemma.
\end{proof}

Suppose that $\varphi$ is Lipschitz continuous with Lipschitz constant $\L$.
By the compactness of $M$, the Lipschitz continuity of $L$ with respect to $u$ and the fiberwise superlinearity of $L$, 
there exists a constant $C_{\L}$ such that
\begin{equation*}
L(x, u, \dot{x}) \geq \L \|\dot{x}\| + C_{\L} \qquad \forall ~(x,\dot{x})\in TM\text{ and $u$ is bounded}.
\end{equation*}
It follows that for every curve $\gamma:[0,t]\rightarrow M$, we have
\begin{equation*}
\begin{split}
\int_0^t L(\gamma(s), u(\gamma(s),s), \dot{\gamma}(s))\ ds &\geq \L\  dist(\gamma(0),\gamma(t)) +C_{\L} t\\
&\geq \varphi(\gamma(t)) - \varphi(\gamma(0)) +C_{\L} t.
\end{split}
\end{equation*}
We conclude that
\[
T_t\varphi(x) \geq \varphi(x) + C_{\L} t.
\]

On the other hand, using the constant curve $\gamma_x$ with $\gamma_x(s) = x$ for any $s\in[0,t]$, we obtain
\[
T_t\varphi(x) \leq \varphi(x) + \max_{s\in [0,t]} L(x, T_s\varphi(x), 0) t.
\]

Therefore, we have
\[
\|T_t\varphi - \varphi\|_\infty \leq t \max\{C_{\L}, \max_{x\in M, s\in[0,t]} L(x, T_s\varphi(x), 0)\}.
\]

Hence, by the semigroup property, we have
\begin{equation*}
\|T_{t_1}\varphi - T_{t_2}\varphi\|_\infty \leq 
\| T_{|t_1-t_2|} \varphi - \varphi \|_\infty \leq |t_1-t_2|\max\{C_{\L}, \max_{x\in M, s\in[0,t]} L(x, u(x,s), 0)\}.
\end{equation*}

In general, we have that for each $\varphi\in C^0(M,\mathbb{R})$ the map $t\mapsto T_t \varphi$ is uniformly continuous.

In order to obtain the Lipschitz property of variational solutions, we need the following crucial observation.
\begin{theorem}\label{uniform bounds for semigroup}
Suppose that the assumptions (H1)--(H6) hold.
For every $\varphi\in C^0(M,\mathbb{R})$, there exists
a $K\in\mathbb{R}^+$ such that
\[
\|T_t \varphi\|_\infty \leq K \qquad \forall~t\geq0.
\]
\end{theorem}
\begin{proof}
For every $(x,t)\in M\times \mathbb{R}^+$, there exists a calibrated curve
\[
\big(\gamma(s), u(s), p(s)\big) \qquad s\in[0,t]
\]satisfying $\gamma(t) = x$ such that
\[
u(x, t) \equiv T_t\varphi (x) = \varphi(\gamma(0)) + \int_0^t L\big( \gamma(s), u(s), \dot{\gamma}(s)\big)\ ds.
\]

(I) We will first show that $T_t\varphi(x)$ is bounded from below. Suppose $u(x,t) < \a$.
Then, there are two cases along the curve $\gamma$:
\begin{itemize}
\item [(1)] There exists a $\tau_0\in [0,t)$ such that $u(\tau_0) = \a$ and $u(\tau)< \a$ when $\tau> \tau_0$.
  Consequently,
  \[
  u(x,t) = u(\tau_0) + \int_{\tau_0}^t L\big( \gamma(\tau), u(\tau), \dot{\gamma}(\tau)\big)\ d\tau.
  \]
  By (H5) and point (2) of Proposition \ref{property of the Lagrangian}, we have the following estimates:
  \begin{equation*}
  \begin{split}
  u(x, t) &\geq u(\tau_0) + \int_{\tau_0}^t L\big( \gamma(\tau), \a, \dot{\gamma}(\tau)\big)\ d\tau\\
  &\geq \a + \min_{x,y\in M} h_{\a}^{t-\tau_0}(x,y).
  \end{split}
  \end{equation*}where $h_{\a}^s(x, y)$ is the barrier function for the autonomous Lagrangian $L(x, \a, \dot{x})$.

\item [(2)]For every $\tau\in[0,t]$, we have $u(\tau) < \a$. Hence, we obtain:
\begin{equation*}
\begin{split}
u(x,t) &= u(0) + \int_0^t L\big( \gamma(\tau), u(\tau), \dot{\gamma}(\tau)\big)\ d\tau\\
       &\geq \min_{x\in M} \varphi(x) + \min_{x,y\in M} h_{\a}^t (x,y).
\end{split}
\end{equation*}
\end{itemize}
From Mather theory, we know that $u(x,t)$ is uniformly bounded independent of $t$ for both cases.
Take $K_1 = \min\{\a, \a + \min_{x,y\in M} h_{\a}^{t-\tau_0}(x,y), \min_{x\in M} \varphi(x) + \min_{x,y\in M} h_{\a}^t (x,y)\}$
and we obtain the uniform lower bound of $u(x,t)$ for any $(x,t) \in M\times \mathbb{R}^+$.

(II) We now show the uniform upper bound of $u(x,t)$. For every $(x,t) \in M\times\mathbb{R}^+$ and a given point $x_0\in M$,
one can find a minimizing curve $\Gamma$ for the autonomous Lagrangian $L(x,\a, \dot{x})$ such that $\Gamma(0)=x_0, \Gamma(t) =x$.
Suppose $u(x,t) < \a$. Then, there are two cases along the curve $\Gamma$:
\begin{itemize}
\item [(1)]There exists a $\tau_0\in [0,t)$ such that $u(\Gamma(\tau_0),\tau_0) = \a$ and $u(\Gamma(\tau), \tau)> \a$ when $\tau> \tau_0$.
So we have the estimates:
\begin{equation*}
\begin{split}
u(x,t) &\leq u(\Gamma(\tau_0), \tau_0) + \int_{\tau_0}^t L\big( \Gamma(\tau), u(\Gamma(\tau),\tau), \dot{\Gamma}(\tau)\big)\ d\tau\\
       &\leq \a + \int_{\tau_0}^t L\big( \Gamma(\tau), \a, \dot{\Gamma}(\tau)\big)\ d\tau\\
       &= \a + h_{\a}^{t-\tau_0}(\Gamma(\tau_0), x).
\end{split}
\end{equation*}

To get the upper bound for $u(x,t)$, it suffices to prove that $h_{\a}^{t-\tau_0}(\Gamma(\tau_0), x)$ is bounded. In fact,
due to the properties of the barrier function in Mather theory, when $\tau>1$, one can find $z$ such that
\begin{equation*}
h_{\a}^{\tau}(x,y) = h_{\a}^{\tau_0}(x, z) + h_{\a}^{\tau- \tau_0}(z, y).
\end{equation*}
We can have the fact that there exists a $A>0$ such that $|h_{\a}^{\tau} (x,y) | \leq A$ for every $\tau>1$.
Likewise, there exists a $B>0$ such that $|h_{\a}^{\tau} (x,y) | \leq B$ for every $\tau>\frac{1}{2}$.

Hence, we obtain
\begin{equation*}
\begin{split}
|h_{\a}^{\tau_0}(x,z)| &\leq A+B \qquad \text{when $\tau_0$ is small enough};\\
|h_{\a}^{\tau-\tau_0}(z,y)| &\leq A+B \qquad \text{when $\tau-\tau_0$ is small enough}.
\end{split}
\end{equation*}
Notice that we have $\Gamma(\tau_0)$ such that the following equality holds:
\begin{equation*}
h_{\a}^{t}(x_0,x) = h_{\a}^{\tau_0}(x_0, \Gamma(\tau_0)) + h_{\a}^{t- \tau_0}(\Gamma(\tau_0), x),
\end{equation*}which shows that $u(x,t)$ is bounded from above.

\item [(2)] For every $\tau\in[0,t]$, we have $u(\tau) > \a$. Hence, we obtain:
\begin{equation*}
\begin{split}
u(x,t) &\leq \varphi(x_0) + \int_0^t L\big( \Gamma(s), u(\Gamma(s), s), \dot{\Gamma}(s)\big)\ ds\\
       &\leq \varphi(x_0) +  \int_0^t L\big( \Gamma(s), \a, \dot{\Gamma}(s)\big)\ ds\\
       &\leq \max_{x\in M}\varphi(x) + h_{\a}^t(x_0,x),
\end{split}
\end{equation*}which is bounded.
\end{itemize}
This completes the proof of the theorem.
\end{proof}

Denote $\phi^t = \Phi^t\circ \mathcal{L}^{-1}: TM\times \mathbb{R}\circlearrowleft$
 where $\Phi^t$ is the phase flow for \eqref{characteristic equ} and $\mathcal{L}$ is the Legendre transform. We have
\begin{lemma}\label{estimate before a priori compactness}
For any $\kappa>0$, there exists a $A>0$ such that if $x\in M, |u|\leq K, |\dot{x}|\geq A$, we have
\begin{equation*}
\|\dot{x}(t)\| \geq \kappa \qquad \forall ~t\in [-1, 1],
\end{equation*}where $(x(t), u(t), \dot{x}(t)) = \phi^t(x,u,\dot{x})$.
\end{lemma}
\begin{proof}
We suppose by contradiction that there exist $\kappa_0$ such that for any $n\in \mathbb{Z}^+$,
if $x_n\in M, |u_n|\leq K, |v_n|\geq n$, there exists an $s_n\in [-1,1]$ such that $|v_n(s_n)|\leq \kappa_0$.
One can choose a subsequence $\{s_{n_i}\}$ of $\{s_{n}\}_{n\in\mathbb{Z}^+}$ such that
\begin{equation*}
s_{n_i}\rightarrow s_0\in [-1,1],~x_{n_i}(s_{n_i})\rightarrow x_0\in M,~u_{n_i}(s_{n_i})\rightarrow u_0, ~v_{n_i}(s_{n_i})\rightarrow v_0.
\end{equation*}
Due to Theorem \ref{uniform bounds for semigroup} and the assumption, we have $|u_0|\leq K$ and $|v_0|\leq \kappa_0$.
This contradicts with the assumption of $|v_{n_i}|
\geq n_i$ by the completeness of $\Phi^t$.
\end{proof}

We will show the following lemma of a priori compactness.
\begin{lemma}[A Priori Compactness]
There exists a $A>0$ such that for every calibrated curve $(\gamma(t), u(t), p(t))$, we have
\begin{equation*}
|\dot{\gamma}(t)| \leq A \qquad \text{when $t>2$}.
\end{equation*}
\end{lemma}
\begin{proof}
Suppose by contradiction that for any $n$, there exists a $t_n>2$ and $x_n$ such that
the calibrated curve
\[
\big( \gamma_n(s), u_n(s), \dot{\gamma}_n(s)\big) \qquad s\in[0,t_n] \quad\text{ with } \gamma_n(t_n) =x_n
\]
satisfies $|\dot{\gamma}_n(t_n)|\geq n$.

Take $K$ in Theorem \ref{uniform bounds for semigroup} such that $|u(x,t)|\leq K$.
Due to the superlinear growth of $L(x, K, \dot{x})$,
there exists a $\kappa>0$ such that when $\dot{x}>\kappa$, we have
\[
L(x, K, \dot{x}) \geq 5K \qquad \forall ~x\in M.
\]

When $n$ large enough, applying Lemma \ref{estimate before a priori compactness}, we have
\[
|\dot{\gamma}_n(s)| \geq \kappa \qquad s\in [t_n-1, t_n].
\]

Consequently, $s>1$ and $L(\gamma_n(s), K, \dot{\gamma}_n(s)) \geq 5K$ for every $s\in [t_n-1, t_n]$.
Hence, we estimate
\begin{equation*}
\begin{split}
|u(x_n,t_n) - u(\gamma_n(t_n-1), t_n-1)|& = |u_n(t_n) - u_n(t_n -1)|\\
                          & = \left|\int_{t_n-1}^{t_n} L(\gamma_n(\tau), u_n(\tau), \dot{\gamma}_n(\tau))\ d\tau\right|\\
                          &\geq \left|\int_{t_n-1}^{t_n} L(\gamma_n(\tau), K, \dot{\gamma}_n(\tau))\ d\tau\right|\\
                          &\geq 5K,
\end{split}
\end{equation*}which is a contradiction with the fact that $|u(x,t)|\leq K$ when $t>1$. It ends the proof of the lemma.

\end{proof}

With above preliminary results and using a similar argument as in \cite{Fathi'08}, 
it is not difficult to show a variant of Fleming Lemma in our context.
\begin{theorem}\label{equiLipschitz}
For every $\varphi\in C^0(M,\mathbb{R})$, the family of functions $T_t \varphi$ with $t\geq 1$ is equi-Lipschitz.
\end{theorem}

\subsection{Convergence of the solution semigroup}
The goal of this section is to prove the following convergence theorem.
\begin{theorem}\label{convergence}
Suppose that $H$ is a $C^\infty$ function satisfying the hypotheses (H1)--(H6).
Let $T_t: C^0(M, \mathbb{R}) \rightarrow C^0(M,\mathbb{R})$ be the associated solution semigroup.

Then, for each $\varphi\in C^0(M,\mathbb{R})$, the limit of $T_t\varphi$, as $t\rightarrow +\infty$, exists.
Moreover, let us denote $u_\infty$ this limit and we obtain that $u_\infty$ satisfies \eqref{Dirichlet problem}.
\end{theorem}

Before proving the theorem, we first recall some crucial facts in the next.
Due to Theorem~\ref{solution semigroup generate weak KAM solution}
we know that $u(x,t) = T_t\varphi(x)$ is a variational solution of \eqref{Cauchy problem}.

Hence, by the definition of variational solutions,
there exists a $C^1$ curve $\gamma_t: [0, t]$ with $\gamma_t(t)=x$ such that
\begin{equation*}
\begin{split}
T_t\varphi(x) - T_s\varphi(\gamma_t(s)) &= \int_{s}^{t} L(\gamma_t(\tau),
T_\tau \varphi(\gamma_t(\tau)), \dot{\gamma_t}(\tau))d\tau \\
&=\inf_{\substack{\xi(s)=\gamma_t(s),\xi(t)=x\\
\xi\in C^{ac}([s,t],M)}}\int_{s}^{t} L(\xi(\tau),
T_\tau\varphi(\xi(\tau)), \dot{\xi}(\tau))d\tau \quad \forall~ 0\leq s<t.
\end{split}
\end{equation*}

We need to investigate the long time behavior of the energy $H$ on the calibrated curve $\gamma_t$ as $t\rightarrow +\infty$.
Note that due to Theorem \ref{regularity result}, for every $t>0$, the calibrated curve
\begin{equation}\label{c-curve}
 \big(\gamma_t(s),u_t(s)=T_s\varphi(\gamma_t(s)),p_t(s)=\frac{\partial
L}{\partial \dot{x}}(\gamma_t(s),u_t(s),\dot{\gamma_t}(s))\big)
\end{equation} is also a characteristic curve of \eqref{Cauchy problem}.

Along the characteristics, for every $s\in [0,t]$ we calculate:
\begin{equation}
\begin{split}
\frac{dH}{ds}(\gamma_t(s),u_t(s),p_t(s)) &= \frac{\partial H}{\partial x}\
\dot{\gamma_t}(s) + \frac{\partial H}{\partial u}\ \dot{u}_t(s) +
\frac{\partial H}{\partial p}\ \dot{p}_t(s)\\
& = -\frac{\partial H}{\partial u}(\gamma_t(s),u_t(s),p_t(s))\
H(\gamma_t(s),u_t(s),p_t(s)).
\end{split}
\end{equation}

Let $H_t(s) = H(\gamma_t(s),u_t(s),p_t(s))$ and we have


\begin{itemize}
\item [(1)] $H_t(s)$ is a decreasing function of $s$ if $H_t(0) >0$;
\item [(2)] $H_t(s)$ is an increasing function of $s$ if $H_t(0) < 0$;
\item [(3)] $H_t(s)=0$ if $H_t(0) = 0$.
\end{itemize}

We now give an energy estimate for some initial time $s_0\in [0,1]$ depending on $(x,t)\in M\times\mathbb{R}^+$.
As a corollary, we have uniform bounds for the energy $H_t(s)$ when $1\leq s \leq t$.
\begin{lemma}\label{initial energy estimate}
For every $x\in M$, $t\geq 1$, there exists an $s_0= s_0(x,t) \in[0,1]$ such that the calibrated curve given by \eqref{c-curve}
satisfies
\[
|H(\gamma_t(s_0),u_t(s_0),p_t(s_0))| \leq H_0.
\]
\end{lemma}
\begin{proof}
Moreover, assuming $t\geq 1$, we first observe by the continuity of $u$ that there exists an $A>0$ such that
\[
|u(x, s)| \leq A \qquad \forall~ s\in[0,1].
\]

Secondly, we claim that there exists a $B>0$ such that for every $t\geq 1$, there exists an $s_0\in [0,1]$ satisfying
\[
|p_t(s_0)| \leq B.
\]
In fact, we suppose by contradiction that for every $B>0$, there is a $t\geq 1$ such that
\[
|p_t(s)| \geq B \qquad \forall ~s\in [0,1].
\]
Hence, by the superlinear growth of $L(x, u,\dot{x})$ with respect to $\dot{x}$, we know
$L(\gamma_t(s), u_t(s), \dot{\gamma}_t(s))$ is unbounded for any $s\in [0,1]$.
This contradicts the fact that $u(x,s)$ is bounded, which shows the claim.

Consequently, let us take
\[
H_0 = \max_{\substack{x\in M\\ |u|\leq A\\|p|\leq B}} H(x,p,u)
\]which is independent of $x$ and $t$,
and then for every $t\geq 1$, there exists a point $s_0\in [0,1]$
of the calibrated curve $(\gamma_t(s),u_t(s),p_t(s))$ such that
\[
|H(\gamma_t(s_0),u_t(s_0),p_t(s_0))| \leq H_0.
\]
\end{proof}

Let us now show the following proposition which is a key ingredient in the proof of Theorem \ref{convergence}.
\begin{prop}\label{key ingredient for convergence}
If the limit energy of $H_{t_n}(t_n)$ exists as $t_n\rightarrow+\infty$, we have $\lim_{t_n\rightarrow +\infty}H_{t_n}(t_n)\leq 0$.
\end{prop}
\begin{proof}
We first observe that from Lemma \ref{initial energy estimate},
one can choose a strictly increasing sequence $t_n\rightarrow +\infty$ as $n\rightarrow +\infty$ such that
\[
\lim_{n\rightarrow +\infty} H\big(\gamma_{n}(t_n),u_{n}(t_n),p_{n}(t_n)\big) = a ,
\]where $(\gamma_n, u_n):[0,t_n]\rightarrow M\times\mathbb{R}$ is a calibrated curve of \eqref{Cauchy problem} with $\gamma_n(t_n) =x$
and $u_n(s)=u(\gamma_n(s),s)$ is a variational solution of \eqref{Cauchy problem}.

We suppose by contradiction that $a>0$.

We notice that $\gamma_n$ is a characteristic curve by Theorem \ref{regularity result} and therefore $C^2$ by the characteristic equation \eqref{characteristic equ}.

In the next, we will use the ``diagonal sequence trick'' for the sequence $(\gamma_n,\dot{\gamma}_n)$.
For every $N\in \mathbb{N}$, by Ascoli-Arzela theorem, we can find a subsequence $(\gamma_{n_k},\dot{\gamma}_{n_k})$
such that
\[
(\gamma_{n_k}(s),\dot{\gamma}_{n_k}(s)) \rightarrow (\gamma_\infty^N(s), \dot{\gamma}_\infty^N(s)) \qquad \forall ~s\in [0, N].
\]
For typographical simplicity, we still denote $(\gamma_n,\dot{\gamma}_n)$ this subsequence $(\gamma_{n_k},\dot{\gamma}_{n_k})$.

Likewise, we can find a subsequence $(\gamma_{n_k},\dot{\gamma}_{n_k})$ such that
\[
(\gamma_{n_k}(s),\dot{\gamma}_{n_k}(s)) \rightarrow (\gamma_\infty^{N+1}(s), \dot{\gamma}_\infty^{N+1}(s)) \qquad \forall ~s\in [0, N+1].
\]
Note that $(\gamma_\infty^N(s), \dot{\gamma}_\infty^N(s)) =(\gamma_\infty^{N+1}(s), \dot{\gamma}_\infty^{N+1}(s))$ for every $s\in [0, N]$.

We continue this procedure and we finally obtain a curve $(\gamma_\infty(s), \dot{\gamma}_\infty(s))$ for $s\in [0,+\infty)$.
Using the continuity of $u(x,t)$ with respect to $x$, one can have a limit point $u_\infty$ of $u_n$ at the same time.
Therefore, the limit point $(\gamma_\infty, u_\infty, p_\infty)$ of $(\gamma_n, u_n, p_n)$ are obtained.

Hence, by our assumption, we have
\begin{equation}\label{energy assumption}
\lim_{s\rightarrow +\infty} H\big(\gamma_{\infty}(s),u_{\infty}(s),p_{\infty}(s)\big) = a>0 ,
\end{equation}
and $H\big(\gamma_{\infty}(s),u_{\infty}(s),p_{\infty}(s)\big)$ is a decreasing function of $s$. That is, for every $\eta>0$, there is an $S$ such that
\[
H\big(\gamma_{\infty}(s),u_{\infty}(s),p_{\infty}(s)\big)>a-\eta\qquad \forall ~s\geq S.
\]

For the proof of Proposition \ref{key ingredient for convergence},
we need to introduce a common fixed point of $T_t$ to control the energy of $(\gamma_\infty(s), u_\infty(s), p_\infty(s))$ as $s$ goes to $+\infty$.
\begin{lemma}\label{common fixed point existence}
For every $\varphi\in C^0(M,\mathbb{R})$, let us define
\begin{equation}
\bar{u} = \limsup_{t\rightarrow\infty} T_t\varphi.
\end{equation}
Moreover, the limit of $T_t \bar{u}$ exists for $t\rightarrow +\infty$
and this limit is a common fixed point of $T_t$.
\end{lemma}
\begin{proof}
To prove the Lemma \ref{common fixed point existence}, we will first
claim $T_t \bar{u}\leq \bar{u}$ holds for any $t\geq0$.
In fact, due to the definition of limsup, we have, for every
$\epsilon>0$, there exists an $S\in\mathbb{R}^+$ such that
\begin{equation}
T_s\varphi + \epsilon \geq \bar{u}\qquad \forall ~s\geq S.
\end{equation}
By the non-expansiveness and monotonicity of $T_t$, we get
\begin{equation}
T_t\circ T_s \varphi + \epsilon \geq T_t(T_s\varphi + \epsilon) \geq
T_t \bar{u}.
\end{equation}
Taking limsup of the above inequality as $s\rightarrow\infty$, we
obtain
\begin{equation}
\bar{u}+\epsilon \geq T_t \bar{u}.
\end{equation}
Since $\epsilon$ is arbitrary, we have $T_t \bar{u} \leq
\bar{u}$. Hence, by the monotonicity of $T_t$, it is easy to
see that $T_t \bar{u}$ is decreasing in $t$ and so has a
limit point as $t\rightarrow\infty$. We denote
\[\bar{u}_0 = \lim_{t\rightarrow\infty} T_t\bar{u}.\]

Since $T_t\varphi(x)$ is equi-Lipschitz with $t\geq 1$, by Theorem \ref{uniform
bounds for semigroup} and Ascoli-Arzela Theorem, we have
$\bar{u}_0\in C^0(M)$, which is a common fixed point of
$T_t$. This completes the proof of the lemma.
\end{proof}

Due to the property of $\bar{u}_0$, we know for every $\epsilon>0$ there exists a $T\in \mathbb{R}^+$ such that
\[
T_t\varphi < \bar{u}_0 + \epsilon \qquad \forall ~t\geq T.
\]

So, by the proper condition, we obtain for $s>T$ large enough that
\begin{equation}\label{energy comparison}
\begin{split}
H\big(\gamma_\infty(s), u_\infty(s), p_\infty(s)\big) &\leq H\big( \gamma_\infty(s), \bar{u}_0(\gamma_\infty(s)) + \epsilon, p_\infty(s) \big)\\
L\big(\gamma_\infty(s), u_\infty(s), \dot{\gamma}_\infty(s)\big) &\geq L\big( \gamma_\infty(s), \bar{u}_0(\gamma_\infty(s)) + \epsilon, \dot{\gamma}_\infty(s) \big).
\end{split}
\end{equation}

Consequently, choosing $\epsilon,\eta>0$ satisfying $\lambda \epsilon + \eta < \frac{a}{2}$,
by (H4), \eqref{energy comparison} and \eqref{energy assumption}, we get
\begin{equation}\label{energy comparison2}
\begin{split}
H\big( \gamma_\infty(s), \bar{u}_0(\gamma_\infty(s)), p_\infty(s) \big)&\geq H\big( \gamma_\infty(s), \bar{u}_0(\gamma_\infty(s)) + \epsilon, p_\infty(s) \big) - \lambda \epsilon\\
&\geq a-\eta - \lambda \epsilon> \frac{a}{2}
\end{split}
\end{equation}where $s\geq \max\{S, T\}$ is large enough.

On the other hand, using point (2) of Proposition \ref{property of the Lagrangian} and \eqref{energy comparison}, we have
\begin{equation}\label{action comparison}
\begin{split}
&\int_T^s L\big( \gamma_\infty(\tau), \bar{u}_0(\gamma_\infty(\tau)) , \dot{\gamma}_\infty(\tau) \big)\ d\tau \\
\leq &
\int_T^s L\big( \gamma_\infty(\tau), \bar{u}_0(\gamma_\infty(\tau)) + \epsilon, \dot{\gamma}_\infty(\tau) \big) + \lambda \epsilon\ d\tau \\
\leq & \int_T^s L\big(\gamma_\infty(\tau), u_\infty(\tau), \dot{\gamma}_\infty(\tau)\big)\ d\tau + \lambda\epsilon (s-T)\\
= & u(\gamma_\infty(s), s) - u(\gamma_\infty(T), T) + \lambda\epsilon (s-T)
\end{split}
\end{equation}where $s\geq \max\{S, T\}$ is large enough.
Note that the terms $u(\gamma_\infty(s), s)$ and $u(\gamma_\infty(T), T)$ are bounded by Theorem \ref{uniform bounds for semigroup}.

We now argue that \eqref{action comparison} contradicts \eqref{energy comparison2}
and therefore we obtain $a\leq 0$. This finishes the proof of the proposition.

All what remains to prove is the following lemma.
\begin{lemma}\label{energy comparison on Lagrangian graph}
Let $\overline{H}(x, p) = H(x, \bar{u}_0(x), p)$. For every $\delta> 0$, there exist $\Lambda>0, B>0$ such that
when $\gamma: [0,s] \rightarrow M $ satisfies $\overline{H}(\gamma(\tau), p(\tau)) > \delta$ we have
\[
\int_0^s \overline{L}(\gamma(\tau), \dot{\gamma}(\tau))  \ d\tau \geq \Lambda s - B
\]where $\overline{L}$ is the associated Lagrangian of $\overline{H}$.
\end{lemma}
\begin{proof}
From Lemma \ref{common fixed point existence}, we know that $\bar{u}_0(x)$ is variational solution of \eqref{Cauchy problem}, i.e.,
\[
H(x, \bar{u}_0(x), d_x \bar{u}_0(x) ) = 0.
\]
This means $0$ is a critical value of the Hamiltonian $\overline{H}$.

Let us denote by $\mathscr{D}$ the set of all the differentiable points of $\bar{u}_0$ in $M$.
Due to the Lipschitz property of $\bar{u}_0$, one can define
\begin{equation*}
\widetilde{L}(x,\dot{x}) = \left\{\!\!\!
  \begin{array}{rl}
   &\overline{L}(x, \dot{x}) - \langle d_x\bar{u}_0(x), \dot{x}\rangle  \qquad x\in \mathscr{D},\\
   &\inf \left\{ \liminf_{\mathscr{D}\ni x_n\rightarrow x} [\overline{L}(x, \dot{x}) - \langle d_x\bar{u}_0(x), \dot{x}\rangle]
    \right\} \qquad x\notin\mathscr{D}.
  \end{array}
\right.
\end{equation*}

Denote
$\Gamma =\left\{ \big(x, \mathcal{L}(d_x \bar{u}_0(x))\big)~:~x\in \mathscr{D}  \right\}$
where $\mathcal{L}$ is the Legendre transform associated with $\overline{L}$. Therefore, we obtain the following facts
\begin{itemize}
\item [(1)] $\widetilde{L}\big|_\Gamma = 0$;

\item [(2)] $\frac{\partial \widetilde{L}}{\partial \dot{x}}\big|_\Gamma
= \frac{\partial \overline{L}}{\partial \dot{x}}(x,\mathcal{L}(d_x \bar{u}_0(x))) - d_x \bar{u}_0(x) = 0$.
\end{itemize}

Take $x\in \pi_1 \Gamma\subseteq M$ where $\pi_1$ is the projection from $TM$ to $M$,
and then for every $\dot{x}\in T_x M$ we have from above facts (1), (2) and (H1)
\[
\widetilde{L} (x, \dot{x}) \geq M \|\dot{x} - \mathcal{L} (d_x \bar{u}_0(x))\|^2.
\]

Consequently, denoting $\overline{\Gamma}$ the closure of $\Gamma$ in $TM$, we have
\begin{equation}\label{Lagrangian graph}
\widetilde{L}(x, \dot{x})
\left\{\!\!\!
  \begin{array}{rl}
   &=0  \qquad (x,\dot{x})\in \Gamma,\\
   &>0  \qquad (x,\dot{x}) \notin \Gamma,
  \end{array}
\right.
\end{equation}
and
\begin{equation}
\overline{H}\big|_{\overline{\Gamma}} = 0.
\end{equation}

When $\overline{H}(\gamma(\tau), d_x \bar{u}_0(\gamma(\tau))) > \delta > 0$, one can find a $\Delta>0$ such that
$dist((\gamma(\tau), \dot{\gamma}(\tau)), \overline{\Gamma}) > \Delta$. Then, by \eqref{Lagrangian graph}, we can find a $\Lambda = \Lambda(\delta)> 0$ such that
$\widetilde{L} (\gamma(\tau), \dot{\gamma}(\tau)) > \Lambda.$

Thus, we can calculate
\begin{equation}
\begin{split}
\int_0^s \overline{L} (\gamma(\tau), \dot{\gamma}(\tau))  \ d\tau & \geq \Lambda s + \int_0^s \langle d_x \bar{u}_0(\gamma(\tau)), \dot{\gamma}(\tau) \rangle \ d\tau\\
&\geq \Lambda s - B,
\end{split}
\end{equation}where $B= 2\max \|\bar{u}_0\|$. This concludes the lemma.
\end{proof}

To check that \eqref{action comparison} contradicts \eqref{energy comparison2}, we just choose $\epsilon>0$ small enough such that $\lambda\epsilon< \Lambda$. So by Lemma \ref{energy comparison on Lagrangian graph}, we have
\[
\Lambda (s-T) \leq B + \lambda\epsilon(s-T) +  u(\gamma_\infty(s), s) - u(\gamma_\infty(T), T)
\]
which is a contradiction when $s-T$ is large enough.
\end{proof}

\begin{proof}[Proof of Theorem \ref{convergence}]
We now continue with the proof of Theorem \ref{convergence}.
Due to Theorem \ref{uniform bounds for semigroup} and Lemma \ref{equiLipschitz},
we know from Ascoli-Arzela theorem that
there exist a strictly increasing sequence $t_n\rightarrow +\infty$ and a Lipschitz function $u_\infty$
such that $T_{t_n} \varphi \rightarrow u_\infty$ uniformly.

From Proposition \ref{key ingredient for convergence}, we obtain that
\[
H(x, u_\infty(x), d_x u_\infty (x)) \leq 0
\] for almost all $x\in M$.

Let us denote $\overline{H}(x, p) = H(x, u_\infty (x), p)$.
For every continuous piecewise $C^1$ curve $\gamma:[t_1, t_2]\rightarrow M$ with $0\leq t_1 < t_2$,
we have
\begin{equation*}
\begin{split}
u_\infty(\gamma(t_2) ) - u_\infty(\gamma(t_1)) &\leq \int_{t_1}^{t_2} \overline{L} (\gamma(s), \dot{\gamma}(s)) \ ds\\
& = \int_{t_1}^{t_2} L (\gamma(s), u_\infty (\gamma(s)), \dot{\gamma}(s))\ ds,
\end{split}
\end{equation*}where $\overline{L}$ is the associated Lagrangian of $\overline{H}$.

Hence, $u_\infty \leq T_t u_\infty$ for each $t\geq 0$.
By the monotonicity of $T_t$, we know that $T_t u_\infty$ is increasing in $t$.

Let us denote $s_n = t_{n+1} - t_n$, which is a sequence goes to $+\infty$ as $n\rightarrow +\infty$. Therefore, we have
\begin{equation}
\begin{split}
\|T_{s_n}u_\infty - u_\infty\|_\infty &\leq \|T_{s_n}u_\infty - T_{s_n +t_n}\varphi\|_\infty+\| T_{t_{n+1}} \varphi - u_\infty\|_\infty\\
& \leq \|u_\infty -  T_{t_n}\varphi\|_\infty+\| T_{t_{n+1}} \varphi - u_\infty\|_\infty
\end{split}
\end{equation}by the non-expansiveness of $T_t$.
This shows that $\lim_{n\rightarrow+\infty}T_{s_n} u_\infty = u_\infty$,
which asserts that $u_\infty$ is common fixed point for $T_t$ with $t\geq 0$.

To show that the limit of $T_t\varphi$ exists as $t\rightarrow +\infty$, it then remains to prove $T_t \varphi \rightarrow u_\infty$ as $t\rightarrow+\infty$.

With this aim , we estimate that
\begin{equation*}
\|T_t \varphi - u_\infty\|_\infty = \|T_{t-t_n} \circ T_{t_n} \varphi - T_{t-t_n}u_\infty\|_\infty \leq \|T_{t_n} \varphi - u_\infty\|_\infty,
\end{equation*} when $t>t_n$.
This finishes the proof of the theorem since $T_{t_n}\varphi \rightarrow u_\infty$.
\end{proof}


\section*{Appendix}
\begin{proof}[Proof of Theorem \ref{compact lemma}]Let us start by fixing some $t>0$, some $K\in \mathbb{R}$ and some
$u\in C^0(M\times \mathbb{R},\mathbb{R})$.

Since $M\times [0,t]$ is compact in $M\times\mathbb{R}$, the set
$u(M\times[0,t])$ is compact in $\mathbb{R}$. Consequently, there
exists $K_t>0$ such that $|u(x,s)|\leq K_t$ for every $x\in M, ~s\in
[0,t]$.

\textbf{First step:} The set $C^{ac}_K$ is absolutely equicontinuous, i.e.,
for every $\epsilon>0$, there exists $\delta>0$ such that if $0\leq
a_1<b_1\leq a_2\leq b_2\leq \ldots \leq a_n\leq b_n\leq t$ and
$\sum_{i=0}^n b_i-a_i <\delta$, then
\[
\sum_{i=0}^n dist(\gamma(a_i),\gamma(b_i)) <\epsilon \qquad \forall~
\gamma\in C^{ac}_K.
\]

Due to the Lipschitz continuity of $L$ with respect to $u$ and the
fiberwise superlinear growth of $L$, for each $R\geq0$, we can find
$C_R>-\infty$ such that for every $\dot{x}\in T_x M$, we have
\[L(x,u,\dot{x})\geq L(x,0,\dot{x}) - \lambda |u| \geq  R \|\dot{x}\| + C_R - \lambda |u|.\]

Consequently, for every $\epsilon>0$, let us take $R>2\frac{K +
t\lambda K_t -t C_0}{\epsilon}, ~ \delta =
\frac{R\epsilon}{2(C_0-C_R)}$. Suppose we have a finite sequence of
pairwise disjoint sub-intervals $(a_i,b_i)$ of $[0,t]$ as above
satisfies
\[
\sum_{i=1}^n (b_i - a_i )<\delta.
\]
Let $J=\cup_{i=1}^n [a_i,b_i]$. It follows that
\begin{equation*}
\begin{split}
&\quad\sum_{i=1}^n dist\big(\gamma(a_i),\gamma(b_i)\big) \\&\leq
\sum_{i=1}^n
\int_{a_i}^{b_i} \|\dot{\gamma}(s)\|\ ds = \int_J\|\dot{\gamma}(s)\|\ ds \\
&\leq \frac{1}{R} \int_J
[L(\gamma(s),u(\gamma(s),s),\dot{\gamma}(s)) - C_R + \lambda |u(\gamma(s),s)|]\ ds \\
&\leq \frac{A_u(\gamma)}{R} - \int_{[0,t]\setminus J} \frac{L(\gamma(s),u(\gamma(s),s),\dot{\gamma}(s))}{R}\ ds + \frac{\lambda K_t - C_R}{R}\sum_{i=1}^n (b_i - a_i)\\
&\leq \frac{K}{R} + \frac{\lambda K_t - C_0}{R}[t-\sum_{i=1}^n(b_i-a_i)] + \frac{\lambda K_t - C_R}{R}\sum_{i=1}^n (b_i - a_i) \\
&\leq \frac{K+t\lambda K_t - tC_0}{R} + \frac{C_0 - C_R}{R} \sum_{i=1}^n (b_i - a_i) \\
&<\frac{\epsilon}{2} + \frac{\epsilon}{2}=\epsilon,
\end{split}
\end{equation*}
which shows this step.

Hence, for every sequence $\{\gamma_i\}_{i\in\mathbb{N}} \subseteq
C^{ac}_K$, by the compactness of $M$ and Ascoli-Arzela theorem, we
can find a subsequence (still denoted by $\gamma_i$) which converges
to some absolutely continuous curve $\gamma$.

\textbf{Second step:} To complete the proof, we will consist in
showing that $A_u(\gamma)\leq K$. In fact, it suffices to prove that
$A_u$ is lower semicontinuous with respect to the $C^0$-topology. In
the sequel, we will reduce the proof to the case where $M$ is an
open subset of $\mathbb{R}^{k}$ where $k=\text{dim} M$.

Since $\gamma_i$ converges uniformly to $\gamma$, we know that the
set
\[
\mathcal{K} = \gamma([0,t])\cup\cup_{i\in \mathbb{N}}\gamma_i([0,t])
\]
is compact. Using the continuity of $u$, we know
$u(\mathcal{K}\times [0,t])$ is a compact subset of $\mathbb{R}$.

Therefore, by the Lipschitz continuity of $L$ with respect to $u$
and the fiberwise superlinearity of $L$, we can find a constant
$\overline K = C_0 - \lambda K_t$ such that
\[
L(x,u,\dot x)\geq \overline K  \qquad \forall~x\in\mathcal{K},~\dot
x\in T_x M.
\]

If $[a,b]\subseteq [0,t]$, taking $K_0$ as a lower bound of
$L(\gamma_i(\tau),u(\gamma_i(\tau),\tau),\dot \gamma_i(\tau))$ on
$[0,t]\setminus [a,b]$, we have
\[
A_u(\gamma_i|_{[a,b]}) \leq A_u(\gamma_i) - \overline{K} (t-b+a)
\qquad \forall ~i\in\mathbb{N}.
\]
It follows that
\[
\liminf_{i\rightarrow +\infty} A_u(\gamma_i|_{[a,b]}) <+\infty
\qquad \forall ~[a,b]\subseteq [0,t].
\]

By continuity of $\gamma$, we can find a finite sequence $t_1=0<
t_2<\ldots<t_p=t$ and a sequence of coordinate charts $U_1,\ldots,
U_p$ such that \[\gamma([t_{n},t_{n+1}])\subseteq U_n\qquad
n=1,\ldots,p-1.\]

Since $\gamma_i$ converges uniformly to $\gamma$, there exists an
$N_0$ such that when $i\geq N_0$ we have
\[\gamma_i([t_n,t_{n+1}])\subseteq U_n \qquad\text{for } n=1,\ldots
,p.\]

Therefore, it is enough to show in local coordinate charts that
\[
A_u(\gamma|_{[t_n,t_{n+1}]}) \leq \liminf_{i\rightarrow +\infty}
A_u(\gamma_i|_{[t_n,t_{n+1}]}).
\]
because we have $\liminf_{i\rightarrow+\infty}(\alpha_i
+\beta_i)\geq  \liminf_{i\rightarrow+\infty}\alpha_i
+\liminf_{i\rightarrow+\infty}\beta_i$ for sequences of real numbers
$\alpha_i$ and $\beta_i$.

Hence, we do need to prove the lower semi-continuity of $A_u$ in the
case where $M$ is an open subset of $\mathbb{R}^k$.

\textbf{Third step:} We now consider the differentiable point $s\in
(0,t)$ of $\gamma$. It is easy to have the following local estimate
in the case $M = U\subseteq \mathbb{R}^k$:
\begin{lemma}
For any $\epsilon>0$, we have
\begin{equation}\label{local estimate}
\begin{split}
L(x,u(x,\tau), \dot{x}) &\geq
L\big(\gamma(s),u(\gamma(s),s),\dot{\gamma}(s)\big) \\
&\quad + \frac{\partial L}{\partial
\dot{x}}\big(\gamma(s),u(\gamma(s),s),\dot{\gamma}(s)\big)\ (\dot{x}
- \dot{\gamma}(s)) - \epsilon,
\end{split}
\end{equation} provided $(x,\tau)$ and $(\gamma(s),s)$ are
close enough.
\end{lemma}
\begin{proof}
Using the continuity of $u$, let us choose $\eta_0>0$ such that
\[
\overline{V}_{\eta_0} = \{ (x,\tau)\in \mathbb{R}^{k+1} : \|x -
\gamma(s)\| + |\tau-s|\leq \eta_0 \}
\] is a compact subset of $U$.

Since $L$ is Lipschitz continuous with respect to $u$ and fiberwise
superlinear , let $R = \|\frac{\partial L}{\partial
\dot{x}}(\gamma(s),u(\gamma(s),s),\dot{\gamma}(s))\|$, we can find
$C_{R+1}>-\infty$ such that for every $(x,\tau)\in
\overline{V}_{\eta_0}$ and $\dot{x}\in T_x M$, we have
\[L(x,u(x,\tau),\dot{x})\geq L(x,0,\dot{x}) - \lambda |u| \geq  (R+1 )\|\dot{x}\| + C_{R+1} - \lambda K_t.\]

Let \[ \widetilde{K} = L(\gamma(s),u(\gamma(s),s),\dot{\gamma}(s)) -
\frac{\partial L}{\partial
\dot{x}}(\gamma(s),u(\gamma(s),s),\dot{\gamma}(s))(\dot{\gamma}(s)).\]

For $\|\dot{x}\|\geq \widetilde{K} - C_{R+1} + \lambda K_t$, for
$(x,\tau)\in \overline{V}_{\eta_0}$, we obtain
\begin{equation*}
\begin{split}
L(x,u(x,\tau),\dot{x}) &\geq (R+1 )\|\dot{x}\| + C_{R+1} - \lambda K_t\\
&\geq R \|\dot{x}\| + \widetilde{K}\\
&\geq \frac{\partial L}{\partial
\dot{x}}(\gamma(s),u(\gamma(s),s),\dot{\gamma}(s))(\dot{x}) + \widetilde{K}\\
&\geq L\big(\gamma(s),u(\gamma(s),s),\dot{\gamma}(s)\big) \\
&\quad + \frac{\partial L}{\partial
\dot{x}}\big(\gamma(s),u(\gamma(s),s),\dot{\gamma}(s)\big)\ (\dot{x}
- \dot{\gamma}(s)).
\end{split}
\end{equation*}

For $\|\dot{x}\|\leq \widetilde{K} - C_{R+1} + \lambda K_t$, we note
that for $(x,\tau)=(\gamma(s),s)$, we have
\[
L(\gamma(s),u(\gamma(s),s), \dot{x}) \geq
L\big(\gamma(s),u(\gamma(s),s),\dot{\gamma}(s)\big)  +
\frac{\partial L}{\partial
\dot{x}}\big(\gamma(s),u(\gamma(s),s),\dot{\gamma}(s)\big)\ (\dot{x}
- \dot{\gamma}(s))
\]
which follows immediately from the fiberwise convexity of $L$.

Consequently, for $\epsilon>0$, we can find $0<\eta\leq \eta_0$ such
that for every $(x,\tau)\in \overline{V}_\eta$ we have
\[
L(x,u(x,\tau), \dot{x}) \geq
L\big(\gamma(s),u(\gamma(s),s),\dot{\gamma}(s)\big)  +
\frac{\partial L}{\partial
\dot{x}}\big(\gamma(s),u(\gamma(s),s),\dot{\gamma}(s)\big)\ (\dot{x}
- \dot{\gamma}(s)) - \epsilon.
\]
\end{proof}

\textbf{Fourth step:} To apply a standard argument, for every
$C\in\mathbb{R}$, we define the function
\begin{equation}
w_C(s) = \min\left\{~
L\big(\gamma(s),u(\gamma(s),s),\dot{\gamma}(s)\big), C ~\right\}.
\end{equation}
Since $u$ is continuous and $L$ is bounded below, $w_C$ is
integrable of $s$. Therefore, its indefinite integral
$W_C(s)=\int_0^s w_C(\tau)\ d\tau$ is absolutely continuous on
$[0,t]$. We denote
\begin{equation*}
E_C \equiv \left\{~s\in[0,t]:~ \gamma \text{ and } W_C \text{ are
differentiable at } s,~ w_C(s) = \frac{dW_C(s)}{ds}~\right\}
\end{equation*} which has full Lebesgue measure in $[0,t]$.

We apply \eqref{local estimate} with $x=\gamma_i(\tau),\dot{x} =
\dot{\gamma}_i(\tau)$ and we compute
\begin{equation}
\begin{split}
&\liminf_{\delta_1,\delta_2\downarrow0}
\liminf_{n\rightarrow\infty}\frac{1}{\delta_1+\delta_2}\int_{s-\delta_1}^{s+\delta_2}
L\big(\gamma_i(\tau),u(\gamma(\tau),\tau),\dot{\gamma}_i(\tau)\big)\\
\geq& L\big(\gamma(s),u(\gamma(s),s),\dot{\gamma}(s)\big) -\epsilon\\
& +
\liminf_{\delta_1,\delta_2\downarrow0}\liminf_{n\rightarrow\infty}\frac{1}{\delta_1+\delta_2}\int_{s-\delta_1}^{s+\delta_2}
\frac{\partial L}{\partial
\dot{x}}\big(\gamma(s),u(\gamma(s),s),\dot{\gamma}(s)\big)\
(\dot{\gamma}_i(\tau) - \dot{\gamma}(s))\\
=& L\big(\gamma(s),u(\gamma(s),s),\dot{\gamma}(s)\big) -\epsilon
\end{split}
\end{equation}The last equality holds since $\gamma_i$ $C^0$
converges to $\gamma$.

In particular, take $s\in E_C$. This inequality implies for every
$\epsilon>0$, there exists a $\delta_0>0$ such that if
$0<\delta_1,\delta_2\leq \delta_0$, we have
\begin{equation}\label{local estimate2}
\begin{split}
\liminf_{i\rightarrow \infty} \frac{1}{\delta_1+\delta_2}
A_u\big(\gamma_i|_{[s-\delta_1,s+\delta_2}\big) &\geq L(\gamma(s),
u(\gamma(s),s), \dot{\gamma}(s)) - \frac{\epsilon}{2}\\
&\geq \frac{W_C(s+\delta_2) - W_C(s-\delta_1)}{\delta_1+\delta_2} -
\epsilon
\end{split}
\end{equation}

The last step is to extend the local estimate \eqref{local
estimate2} to the global one. In fact, it is not difficult to
construct a countable mutually disjoint sequence
$\{[a_i,b_i]\}_{i\in\mathbb{N}}$ of closed intervals which cover
$E_C$ such that
\begin{equation}
\liminf_{i\rightarrow \infty} \frac{A_u(\gamma_i|_{[a_j,b_j]})}{b_j
- a_j} \geq \frac{W_C(b_j) - W_C(a_j)}{b_j -  a_j} - \epsilon.
\end{equation}
It follows that
\begin{equation}
\liminf_{i\rightarrow \infty} A_u(\gamma_i) \geq W_C(t) -  W_C(0) -
\epsilon t.
\end{equation}
Let $C\uparrow\infty$ and since $\epsilon$ is arbitrary we obtain
\begin{equation}
\liminf_{i\rightarrow \infty} A_u(\gamma_i) \geq A_u(\gamma),
\end{equation}which finishes the proof of the theorem.

\end{proof}


\bibliographystyle{alpha}
\bibliography{weak-KAM}
\end{document}